\documentclass[final,onefignum,onetabnum]{siamart190516}

\usepackage{lipsum}
\usepackage{amsfonts}
\usepackage{graphicx}
\usepackage{epstopdf}
\usepackage{multirow}
\usepackage{algpseudocode}
\ifpdf
  \DeclareGraphicsExtensions{.eps,.pdf,.png,.jpg}
\else
  \DeclareGraphicsExtensions{.eps}
\fi

\usepackage{bm}
\usepackage{pgfplots}
\usetikzlibrary{pgfplots.groupplots}
\usepackage{multirow}
\pgfplotsset{select coords between index/.style 2 args={
    x filter/.code={
        \ifnum\coordindex<#1\fi
        \ifnum\coordindex>#2\fi
    }
}}
\usepackage{cleveref}

\newcommand{\abf}{\ensuremath{\mathbf{a}}}
\newcommand{\bbf}{\ensuremath{\mathbf{b}}}
\newcommand{\cbf}{\ensuremath{\mathbf{c}}}
\newcommand{\dbf}{\ensuremath{\mathbf{d}}}

\newcommand{\fbf}{\ensuremath{\mathbf{f}}}
\newcommand{\gbf}{\ensuremath{\mathbf{g}}}
\newcommand{\zbf}{\ensuremath{\mathbf{z}}}
\newcommand{\Ecal}{\ensuremath{\mathcal{E}}}
\newcommand{\Gcal}{\ensuremath{\mathcal{G}}}
\newcommand{\Lcal}{\ensuremath{\mathcal{L}}}
\newcommand{\Ocal}{\ensuremath{\mathcal{O}}}
\newcommand{\Pcal}{\ensuremath{\mathcal{P}}}
\newcommand{\Tcal}{\ensuremath{\mathcal{T}}}
\newcommand{\Ucal}{\ensuremath{\mathcal{U}}}
\newcommand{\Vcal}{\ensuremath{\mathcal{V}}}
\newcommand{\Zcal}{\ensuremath{\mathcal{Z}}}
\newcommand{\prm}{\ensuremath{\eta}}
\newcommand{\Rbb}{\ensuremath{\mathbb{R}}}
\newcommand{\Pbb}{\ensuremath{\mathbb{P}}}
\newcommand{\Cbb}{\ensuremath{\mathbb{C}}}
\newcommand{\albf}{\bm{\alpha}}
\newcommand{\bebf}{\bm{\beta}}
\newcommand{\bUcal}{\overline{\mathcal{U}}}
\newcommand{\norm}[1]{\left\lVert#1\right\rVert}
\newcommand{\normF}[1]{\left\lVert#1\right\rVert_F}
\newcommand{\seminorm}[1]{|{#1}|}
\newcommand{\Sprm}{\ensuremath{\mathcal{P}}}
\newcommand{\Sprmh}{\ensuremath{\mathcal{P}_h}}
\newcommand{\Vd}[1]{\ensuremath{V_{#1}}}
\newcommand{\nf}{\ensuremath{n}} 
\newcommand{\nr}{\ensuremath{k}} 
\newcommand{\nd}{\ensuremath{m}} 
\newcommand{\nh}{\ensuremath{d}} 
\newcommand{\Nf}{\ensuremath{2\nf}} 
\newcommand{\Nr}{\ensuremath{2\nr}} 
\newcommand{\nt}{\ensuremath{N_t}} 
\newcommand{\ns}{\ensuremath{N_s}} 
\newcommand{\np}{\ensuremath{N_p}} 
\newcommand{\nad}{\ensuremath{N_a}} 
\newcommand{\fs}{\ensuremath{y}} 
\newcommand{\rs}{\ensuremath{z}} 
\newcommand{\hrs}{\ensuremath{z}} 
\newcommand{\redb}{A} 
\newcommand{\deimb}{U} 
\newcommand{\Ham}{\ensuremath{\mathcal{H}}} 
\newcommand{\hn}{\ensuremath{\mathcal{N}}} 
\newcommand{\G}{\ensuremath{{h}}} 
\newcommand{\cv}{\ensuremath{{v}}} 
\newcommand{\dx}{\ensuremath{\mathrm{d}x}}

\DeclareMathOperator*{\argmin}{arg\,min}
\DeclareMathOperator*{\HR}{HR}
\DeclareMathOperator*{\R}{R}
\DeclareMathOperator*{\rank}{rank}
\DeclareMathOperator*{\off}{off}
\DeclareMathOperator*{\on}{on}
\DeclareMathOperator*{\tot}{tot}
\DeclareMathOperator*{\fin}{fin}
\DeclareMathOperator*{\col}{Col}

\newsiamremark{remark}{Remark}
\newsiamremark{hypothesis}{Hypothesis}
\crefname{hypothesis}{Hypothesis}{Hypotheses}
\newsiamthm{claim}{Claim}

\headers{Gradient-preserving hyper-reduction via DEIM}{C. Pagliantini and F. Vismara}

\title{Gradient-preserving hyper-reduction of nonlinear dynamical systems via discrete empirical interpolation\thanks{Submitted to the editors on June 2022.
}}

\author{Cecilia Pagliantini\thanks{Centre for Analysis, Scientific computing and Applications,
			  Eindhoven University of Technology (TU/e),
			  The Netherlands.
  (\email{c.pagliantini@tue.nl}, \email{f.vismara@tue.nl}).}
\and Federico Vismara\footnotemark[2]}

\usepackage{amsopn}
\DeclareMathOperator{\diag}{diag}

\ifpdf
\hypersetup{
  pdftitle={GP-DEIM},
  pdfauthor={C. Pagliantini and F. Vismara}
}
\fi

\begin{document}

\maketitle

\begin{abstract}
This work proposes a hyper-reduction method for nonlinear parametric dynamical systems characterized by gradient fields such as Hamiltonian systems and gradient flows. The gradient structure is associated with conservation of invariants or with dissipation and hence plays a crucial role in the description of the physical properties of the system.
Traditional hyper-reduction of nonlinear gradient fields yields efficient approximations that, however, lack the gradient structure.
We focus on Hamiltonian gradients and we propose to first decompose the nonlinear part of the Hamiltonian, mapped into a suitable reduced space, into the sum of $d$ terms, each characterized by a sparse dependence on the system state. Then, the hyper-reduced approximation is obtained via discrete empirical interpolation (DEIM) of the Jacobian of the derived $d$-valued nonlinear function.
The resulting hyper-reduced model retains the gradient structure and its computationally complexity is independent of the size of the full model. Moreover, \emph{a priori} error estimates show that the hyper-reduced model converges to the reduced model and the Hamiltonian is asymptotically preserved. 
Whenever the evolution of the nonlinear Hamiltonian gradient requires high-dimensional DEIM approximation spaces, an adaptive strategy is performed. This consists in updating the hyper-reduced Hamiltonian via a low-rank correction of the DEIM basis. Numerical tests demonstrate the runtime speedups compared to the full and the reduced models.
\end{abstract}



\section{Introduction}\label{sec:introduction}
We consider parametric dynamical systems where the velocity describing the flow is characterized by nonlinear gradient fields, that is
\begin{equation}\label{eq:ref}
    \dot{\fs}(t,\prm) = X(\fs(t,\prm),\prm)\quad\mbox{for}\;t\geq t_0,
\end{equation}
where $\prm\in\Rbb^{p}$, $p\geq 1$, is a parameter, $\fs:(t_0,\infty)\times\Rbb^{p} \rightarrow\Rbb^{\nf}$ is the state variable, and the velocity field
$X:\Rbb^{\nf}\times\Rbb^{p} \rightarrow\Rbb^{\nf}$ is assumed to
be of the form $X(\fs,\prm)=S\nabla_{\fs}\Ham(\fs,\prm)$, with $S$ in $\Rbb^{\nf\times\nf}$ and $\Ham$ a nonlinear function of the state.
Examples of such dynamical systems are gradient flows, where $S$ is a negative semi-definite matrix, Hamiltonian systems, where $S$ is skew-symmetric, and dissipative Hamiltonian systems, such as port-Hamiltonian systems, where $S=J-R$ is the sum of a skew-symmetric component $J$ and a symmetric positive semi-definite component $R$.
The gradient structure of the velocity field $X$ plays a crucial role in the characterization of the dynamics and of the physical properties of the system since it is associated with dissipation or preservation of quantities such as energy or entropy.

In this work we focus on Hamiltonian dynamical systems and the goal is to develop numerical approximation methods for their efficient solution in large-scale scenarios, i.e. when the number $\nf$ of degrees of freedom is high, and for many instances of the parameter. A major challenge in this task is to preserve the physical properties of the dynamics and, in particular, the Hamiltonian gradient structure of the velocity field.

In the context of large-scale many-query simulations of parametric differential problems, model order reduction (MOR) and reduced basis methods (RBM) have been developed to provide efficient high-fidelity surrogate models on lower-dimensional spaces where the bulk of the dynamics takes place.
Their success notwithstanding, the application of RBM to problems of the form \eqref{eq:ref}
might lead to unstable and qualitatively wrong solution behavior since the gradient structure of the model can be destroyed during dimension reduction.
To address this issue in solving Hamiltonian systems, structure-preserving MOR techniques have been developed in recent years to derive reduced models that retain (at least part of) the geometric structure underlying the dynamics.
The first methods of this class were developed in \cite{LKM03} using a Lagrangian formulation and were later extended to nonlinear parametric Lagrangian systems in \cite{CTB15}; reduced models for canonical Hamiltonian systems on symplectic vector spaces were derived in \cite{PM16} using POD-type strategies and in \cite{AH17} via symplectic greedy algorithms; reduced bases that are symplectic but not orthogonal were proposed in \cite{BBH19,BHR20}; in \cite{GWW17} reduced models were obtained via separate reduced spaces for the generalized positions and momenta; in \cite{UKY21} a double projection of the original problem and of the Hamiltonian gradient into the span of a POD basis is employed; general Hamiltonian systems on Poisson manifolds were tackled in \cite{HP20}; in \cite{SWK21} operator inference techniques are proposed to deal with linear Hamiltonian systems. We refer to \cite{HPRo22} for a detailed overview on structure-preserving MOR for Hamiltonian systems.
Although these methods have led to the successful construction of stable low-dimensional models, little attention has been paid to the efficient treatment of nonlinear operators.
Indeed, in the presence of Hamiltonian functions with general nonlinear dependence on the state, the computational cost of solving the reduced model might still depend on the size of the underlying full model, resulting in simulation times that hardly improve over the original system simulation.
This is a well-known issue in model order reduction and has led to the so-called hyper-reduction methods, which are usually based on approximations of the high-dimensional nonlinear operators using sparse sampling via interpolation among samples.
Traditional hyper-reduction techniques, however, do not preserve gradient fields.
This implies that hyper-reduced models of Hamiltonian systems are no longer Hamiltonian.
As real-world processes tend to be nonlinear,
a lack of efficient and physically compatible dimension reduction of general nonlinearities comes to the fore.

To the best of our knowledge, only a few works have considered hyper-reduction methods aimed at preserving Hamiltonian structures \cite{CTB15,FCA15,CBG16, Wang21,HPR22}.
Although these techniques introduce interesting ideas, they appear tailored to specific problems or lack a rigorous theoretical analysis of the suggested approximations.
In \cite{CTB15} the reduced Hamiltonian is approximated with a Taylor polynomial expansion truncated at the second term. As reported by the authors, the scheme is only effective for asymptotically stable systems and when the expansion is performed around an equilibrium point. These assumptions rule out many important cases and are generally not met by evolution problems stemming from the semi-discretization of partial differential equations.
A cubature approach, the Energy-Conserving Sampling and Weighting (ECSW) scheme, is presented in \cite{FCA15}: the nonlinear vector field obtained from the semi-discretization of a Hamiltonian PDE is approximated with a weighted average of the field components on a coarser mesh. The method is, however, limited to finite element discretization of Hamiltonian PDEs, and requires a very expensive offline phase, especially for parametric problems.
A variation of the discrete empirical interpolation method (DEIM) \cite{BMNP04,CS10} has been proposed in \cite{CBG16}: the nonlinear Hamiltonian gradient is approximated in the space where the DEIM projection is orthogonal. The Hamiltonian structure is preserved since orthogonal projections preserve gradients, but the method does not guarantee that the resulting approximation is accurate.
A suitable combination of DEIM and dynamic mode decomposition (DMD) is proposed in \cite{HPR22} to deal with the nonlinear operators of the Hamiltonian systems resulting from a particle-based discretization of the Vlasov--Poisson kinetic plasma model.
In a recent preprint \cite{Wang21}, DEIM hyper-reduction is applied to a nonlinear vector-valued function $G$ obtained by decomposing the Hamiltonian into the Euclidean product of $G$ with a constant vector. The gradient of the resulting operator provides an approximation of the original Hamiltonian gradient but there is no guarantee on the accuracy of such approximation nor on asymptotic convergence as the DEIM space is enlarged.

In this work we propose a novel gradient-preserving hyper-reduction strategy to construct surrogate models that: (i) are Hamiltonian; (ii) can be solved at a computational cost independent of the size of the full model; (iii) are provably accurate; and (iv) ensure
preservation of the full order Hamiltonian asymptotically.
The proposed approach targets the \emph{reduced} nonlinear operator, in contrast to traditional hyper-reduction techniques where the nonlinear high-dimensional function is approximated first and then the resulting approximation is projected onto a reduced space. In particular,
we first map the full order nonlinear Hamiltonian gradient into the reduced space via a structure-preserving (here symplectic) projection.
Then, the nonlinear part of the reduced Hamiltonian is written as the sum of $\nh$ terms, where $\nh$ is typically of the order of the size of the full model.
Although a similar decomposition is used in \cite{Wang21}, here
we do not approximate the resulting $\nh$-valued nonlinear function but rather the Hamiltonian gradient. This is done by projecting the Jacobian of the nonlinear function using discrete empirical interpolation (DEIM) \cite{BMNP04,CS10}.
We derive \emph{a priori} error estimates
of the error between the full model solution and its hyper-reduced approximation and of the error in the conservation of the Hamiltonian. These results show that the hyper-reduced model is asymptotically as accurate as the reduced model. Since the target accuracy is the one of the reduced model, the hyper-reduction allows to obtain accurate solutions and to ensure exact Hamiltonian conservation with DEIM sizes much smaller than $\nh$ and, hence, at a significantly reduced computational cost.

Whenever the dynamics of the problem does not allow for small DEIM approximations and, hence, enough computational savings, we let the proposed gradient-preserving hyper-reduction change in time. We derive an adaptive approach that extends to gradient fields the
adaptive discrete empirical interpolation method (ADEIM) introduced in \cite{Peher15}.
The DEIM basis is updated in time with the low-rank factor that minimizes a residual of the nonlinear Jacobian at suitably chosen sampling points.
In particular, we extend the method of \cite{Peher15} to updates of general rank and to matrix-valued nonlinear terms. We also derive an algorithm for the optimal choice of the sampling points. A detailed analysis of the computational complexity of the adaptive hyper-reduction shows that, if the adaptation is implemented in an efficient way, the complexity reduction due to a smaller DEIM basis (compared to the non-adaptive algorithm) outweighs the extra computational cost of constructing the update.

The remainder of the paper is organized as follows.
In \Cref{sec:ROM} we introduce Hamiltonian dynamical systems and discuss their symplectic model order reduction. \Cref{sec:HR_model}
concerns the structure-preserving DEIM hyper-reduction of the Hamiltonian gradient.
\Cref{sec:adaptive_DEIM} is devoted to the adaptive gradient-preserving hyper-reduction method and a summary of the scheme is presented in \Cref{alg:SP-ADEIM}.
In \Cref{sec:numexp} the proposed methods are tested on a set of numerical experiments. Some conclusions and open questions are presented in \Cref{sec:conclusions}.

\section{Full order system and its model order reduction}
\label{sec:ROM}
Let $\Tcal:=(t_0,T]\subset\Rbb$ be a temporal interval, let
$\Sprm\subset \mathbb{R}^p$, with $p\geq 1$,
be a compact set of parameters, and let $\Vd{\Nf}\subset\Rbb^{\Nf}$ be the phase space.
For each $\prm\in\Sprm$,
we consider the Hamiltonian dynamical system:
given $\fs^0(\prm)\in \Vd{\Nf}$, find $y(\cdot,\prm)\in C^1(\Tcal;\Vd{\Nf})$ such that
\begin{equation}\label{eq:full_model}
\left\{\begin{aligned}
   & \Dot{\fs}(t,\prm) = J_{\Nf}\nabla_\fs \Ham(\fs(t,\prm),\prm), & t\in\Tcal, \\
   & \fs(t_0,\prm)=\fs^0(\prm), &
\end{aligned}\right.
\end{equation}
where $\fs:\Tcal\times \Sprm\rightarrow \Vd{\Nf}$ is the state variable and the matrix $J_{\Nf}\in\Rbb^{\Nf\times\Nf}$, the so-called canonical Poisson tensor, is defined as
\begin{equation*}
    J_{\Nf}=\begin{pmatrix} 0_\nf & I_\nf \\ -I_\nf & 0_\nf \end{pmatrix}.
\end{equation*}
The function $\Ham:\Vd{\Nf}\times\Sprm\to\Rbb$ is the Hamiltonian of the system and we write it in the following general form
\begin{equation}\label{eq:Ham}
    \Ham(\fs,\prm) = \frac{1}{2}\fs^{\top}L(\prm)\fs+\fs^\top f(\prm)+\Gcal(\prm) +\hn(\fs,\prm),
\end{equation}
where
$L(\prm)\in\Rbb^{\Nf\times\Nf}$ is a symmetric positive semi-definite matrix, $f(\prm)\in\Rbb^\nf$, $\Gcal(\prm)\in\Rbb$ and $\hn(\cdot,\prm)$ is a nonlinear function of the state variable. The Hamiltonian is a conserved quantity of system \eqref{eq:full_model}, namely it remains constant along solution trajectories. Moreover, the phase space of Hamiltonian systems has a symplectic geometric structure. Here we focus on symplectic vector spaces: this means that the space $\Vd{\Nf}$ can be endowed with a local basis $\{e_i\}_{i=1}^{\Nf}$ which is symplectic and orthonormal \cite[Chapter 12]{CdS01}, that is	$e_i^\top J_{\Nf} e_j=(J_{\Nf})_{i,j}$
and $(e_i,e_j)=\delta_{i,j}$ for all $i,j=1\ldots,\Nf$,
where $(\cdot,\cdot)$ is the Euclidean inner product.
We refer to \cite[Chapter 10]{MR99} for a detailed introduction to Hamiltonian systems.
To ensure well-posedness of \eqref{eq:full_model} we
assume that, for each parameter $\prm\in\Sprm$, the Hamiltonian gradient $\nabla_{\fs}\Ham(\cdot,\prm)$ is Lipschitz continuous in the Euclidean norm uniformly with respect to time.

Throughout the paper, we use the symbols $\norm{\cdot}$ and $\norm{\cdot}_2$ to denote the vector and matrix $2$-norm, respectively, while $\norm{\cdot}_F$ refers to the Frobenius norm.

\subsection{Symplectic model order reduction}
When the number $\nf$ of degrees of freedom in \eqref{eq:full_model} is large and the system needs to be solved for many instances of the parameter $\prm$,
model order reduction can be used to reduce the complexity of the original system and thus to speed up the 
resulting numerical simulations. Among model order reduction techniques, we consider reduced basis methods, which construct an approximation space of low dimension, the so-called \emph{reduced space}, from a collection of simulation data corresponding to full order solutions at sampled values of time and parameters. A low-dimensional model is then obtained via a suitable projection of the full order dynamics onto the reduced space.

More in details, the state $y$ is approximated as
\begin{equation*}
    y_{rb}(t,\prm)=\sum_{i=1}^{\Nr} z_i(t,\prm) a_i=Az(t,\prm),\qquad\mbox{for all}\; t\in\Tcal,\,\prm\in\Sprm,
\end{equation*}
where $\nr\ll\nf$, $\redb=[a_1\,\ldots\,a_{\Nr}]\in\Rbb^{\Nf\times\Nr}$ is the reduced basis and $\rs(t,\prm)\in\Rbb^{\Nr}$ is the vector of expansion coefficients.
In order to preserve the Hamiltonian structure of the full model, the matrix $\redb$ is enforced to be symplectic, that is, $\redb^{\top}J_{\Nf}\redb=J_{\Nr}$, and orthogonal, that is, $\redb^{\top}\redb=I_{\Nr}$. If $\redb$ is an orthogonal symplectic matrix, then  $\redb^{\top}J_{\Nf}=J_{\Nr}\redb^{\top}$, see e.g. \cite[Lemma 3.3]{PM16}.
Projecting the full order model \eqref{eq:full_model} onto the space $\Vd{\Nr}$ spanned by the columns of $\redb$ yields the reduced model:
for any $\prm\in\Sprm$, find $\rs(\cdot,\prm)\in C^1(\Tcal;\Vd{\Nr})$ such that
\begin{equation}\label{eq:red_model}
\left\{
\begin{aligned}
    & \Dot{\rs}(t,\prm)=J_{\Nr}\nabla_\rs \Ham_r(\rs(t,\prm),\prm), & t\in\Tcal, \\ & \rs(t_0,\prm)=\rs^0(\prm):=\redb^\top \fs^0(\prm), &
\end{aligned}\right.
\end{equation}
where the reduced Hamiltonian is obtained from \eqref{eq:Ham} as
\begin{equation*}
    \Ham_r(\rs,\prm):=\Ham(\redb\rs,\prm)=\frac{1}{2}\rs^\top L_r(\prm)\rs+\rs^\top f_r(\prm)+ \Gcal(\prm) + \hn_r(\rs,\prm),
\end{equation*}
with $L_r(\prm):=\redb^\top L(\prm)\redb\in\Rbb^{\Nr\times\Nr}$, $f_r(\prm):=\redb^\top f(\prm)\in\Rbb^{\Nr}$ and $\hn_r := \hn\circ\redb$.

An orthogonal and symplectic reduced basis can be constructed from a set of
full model solutions using SVD-type strategies, such as cotangent lift, complex SVD, or nonlinear programming \cite{PM16}, or with the symplectic greedy algorithm of~\cite{AH17}.

\section{Hyper-reduction of the Hamiltonian gradient}\label{sec:HR_model}

Assuming the affine separability of the operators $L(\prm)$ and $f(\prm)$ in \eqref{eq:Ham}, the major computational cost in solving the reduced model \eqref{eq:red_model} comes from the nonlinear term $\hn_r$ of the Hamiltonian.
If standard hyper-reduction techniques are applied to the reduced Hamiltonian gradient $\nabla_z\Ham_r$, the resulting approximate function is generally no longer a gradient field, which might lead to unstable or inaccurate approximations. Numerical evidence of this behavior is reported e.g. in \cite{PM16}.
To address this shortcoming we first consider a decomposition of the nonlinear part of the reduced Hamiltonian as
\begin{equation}\label{eq:decomp}
    \hn_r(\rs,\prm) = \sum_{i=1}^\nh \cv_i\G_i(A\rs,\prm)=\cv^{\top}\G(A\rs,\prm),\qquad\mbox{for all}\; \rs\in\Vd{\Nr},\,\prm\in\Sprm,
\end{equation}
where $d\in\mathbb{N}$, $\cv\in\Rbb^\nh$ is a constant vector and $\G:\Vd{\Nf}\times\Sprm\to\Rbb^\nh$.
A decomposition of the Hamiltonian of the form \eqref{eq:decomp} has appeared also in \cite{Wang21} and holds in many cases of interest. For example, the Hamiltonian given by the energy of a system of particles can be written as the sum of the contribution of each particle; or, if \eqref{eq:full_model} results from the semi-discretization of a PDE via a local approximation, e.g. finite element or finite volume schemes, $\G_i$ represents the contribution of the $i$th mesh element to the total Hamiltonian. We refer to \cite{MR99} and \cite[Chapter I]{HLW06} for several examples of Hamiltonian in this form and we will further comment on this aspect in \Cref{rmk:decomp}.

Using the decomposition \eqref{eq:decomp}, the gradient of the nonlinear term $\hn_r$ reads
\begin{equation}\label{eq:redgrad}
    \nabla_\rs \hn_r(\rs,\prm)
    =\redb^{\top}J_{\prm,\G}^{\top}(\redb\rs)\cv,\qquad\mbox{for all}\; \rs\in\Vd{\Nr},\,\prm\in\Sprm,
\end{equation}
where $J_{\prm,\G}(\cdot)\in\Rbb^{\nh\times\Nf}$ is the Jacobian of $\G(\cdot,\prm)$.
The idea is then to use discrete empirical interpolation to approximate the term $J_{\prm,\G}(\redb\cdot)\redb\in\Rbb^{\nh\times\Nr}$, namely the Jacobian mapped to the reduced space. The projected reduced Jacobian is
\begin{equation}\label{eq:DEIMproj}
\Pbb J_{\prm,\G}(\redb\cdot)\redb,
\qquad\mbox{where}\qquad
\Pbb:=\deimb(P^\top\deimb)^{-1}P^\top\in\Rbb^{\nh\times\nh}.
\end{equation}
Here, $\deimb\in\Rbb^{\nh\times\nd}$ is the so-called DEIM basis with $\nd\ll\nh$, and
$P:=[\mathbf{e}_{\beta(1)},\dots,\mathbf{e}_{\beta(\nd)}]$, where $\mathbf{e}_i$ is the $i$th unit vector of $\mathbb{R}^{\nh}$ and $\{\beta(1),\dots,\beta(\nd)\}\subset\{1,\dots,\nh\}$ are interpolation indices \cite{CS10}.
The DEIM projection introduced in \eqref{eq:DEIMproj} can be equivalently seen as an approximation of the nonlinear term $\hn_r$ in \eqref{eq:decomp} by
\begin{equation}\label{eq:hhr_and_gradhhr}
    \hn_{hr}(\rs,\prm):=\cv^\top\Pbb \G(\redb\rs,\prm)\qquad\mbox{for all}\; \rs\in\Vd{\Nr},\,\prm\in\Sprm.
\end{equation}
With the nonlinear term of the Hamiltonian approximated as in \eqref{eq:hhr_and_gradhhr}, the resulting hyper-reduced model reads: for any $\prm\in\Sprm$, find $\hrs(\cdot,\prm)\in C^1(\Tcal;\Vd{\Nr})$ such that
\begin{equation}\label{eq:Hyperred_model}
\left\{
\begin{aligned}
    & \Dot{\hrs}(t,\prm)=J_{\Nr}\nabla_\hrs \Ham_{hr}(\hrs(t,\prm),\prm), & t\in\Tcal, \\
    & \hrs(t_0,\prm) = \hrs^0(\prm):=\redb^\top \fs^0(\prm),
\end{aligned}\right.
\end{equation}
The dynamical system \eqref{eq:Hyperred_model} is still Hamiltonian,
with the hyper-reduced Hamiltonian defined as
\begin{equation*}
    \Ham_{hr}(\hrs,\prm)=\frac{1}{2}\hrs^\top L_r(\prm)\hrs+\hrs^\top f_r(\prm)+ \Gcal(\prm) + \hn_{hr}(\hrs,\prm),
\end{equation*}
and its gradient is given by
$\nabla_\hrs \Ham_{hr}(\hrs,\prm) = L_r(\prm)\hrs + f_r(\prm) + \redb^{\top}J_{\prm,\G}^{\top}(\redb\rs)\Pbb^{\top}\cv$.

\begin{remark}\label{rmk:decomp}
A key role in this approach is played by the decomposition \eqref{eq:decomp}, which is not unique, and by the choice of $\nh$. 
Indeed, let us assume that, for fixed $\hrs\in\Vd{\Nr}$ and $\prm\in\Sprm$, each entry of $\G(\fs,\prm)$, with $\fs:=A\hrs$, depends on $s_1\leq\nf$ components among the first $\nf$ entries of $\fs$ and on $s_2\leq\nf$ components among the last $\nf$ entries of $\fs$. Then, every row of $J_{\prm,\G}$ has at most $s_1+s_2$ non-zero elements. This implies that the evaluation of the gradient \eqref{eq:redgrad} of the reduced Hamiltonian has complexity $O((s_1+s_2)\nh\nr)$, while the cost of computing the gradient of the hyper-reduced Hamiltonian \eqref{eq:hhr_and_gradhhr} is of order $O((s_1+s_2)\nd\nr)$.
Based on this observation, one might be tempted to choose a decomposition \eqref{eq:decomp} with a small $\nh$ and bypass the hyper-reduction entirely.
This choice is however often associated with large values of $s_1$ and $s_2$ since typically $(s_1+s_2)\nh$ is, at least, of the order of $\nf$.
The choice of the decomposition should instead be driven as to maximize the sparsity of the Jacobian matrix $J_{\prm,\G}\in\Rbb^{\nh\times\Nf}$, hence maximizing the efficacy of the hyper-reduction.
We refer to the numerical experiments in \Cref{sec:numexp} for concrete examples of such decomposition.
\end{remark}

\subsection{Construction of the DEIM projection}
\label{sec:ConstrDEIM}
As anticipated in \eqref{eq:DEIMproj}, we propose an approximation that minimizes the DEIM projection error of the reduced Jacobian, $J_{\prm,\G}(\redb\cdot)\redb\in\Rbb^{\nh\times\Nr}$, once the reduced basis $\redb$ has been constructed.
The algorithm of \cite{Wang21} seeks instead to approximate the factor $\G$ in \eqref{eq:decomp},
but this comes with no guarantee on the accuracy of the approximation of the reduced Hamiltonian gradient and, in turn, on the quality of the hyper-reduced model \eqref{eq:Hyperred_model}.

To build the DEIM projection, we compute snapshots in the reduced space $\Vd{\Nr}$ as $\{\redb^\top \fs^{\ell_i}(\prm_j)\}_{i,j}$ for $j=1,\dots,\np$ and $i=1,\ldots,\ns$, where $\fs^{\ell_i}(\prm_j)$ is the full model solution at time $t^{\ell_i}$ and parameter $\prm_j$. We collect the corresponding snapshots of the reduced Jacobian in the matrix
\begin{equation}\label{eq:snaps_JG}
    M_{J}=[J_{\prm_1,\G}(\redb\redb^{\top}\fs^{\ell_1}(\prm_1))\redb\,\ldots\, J_{\prm_{\np},\G}(\redb\redb^{\top}\fs^{\ell_{\ns}}(\prm_{\np}))\redb]\in\Rbb^{\nh\times\Nr \ns \np}.
\end{equation}
The DEIM basis $\deimb\in\Rbb^{\nh\times\nd}$ in \eqref{eq:DEIMproj} is obtained from $M_{J}$ via proper orthogonal decomposition (POD). Observe that the snapshot matrix $M_{J}$ can have a potentially large number of columns: to curb the high computational cost to derive the DEIM basis from $M_{J}$, one could make use of randomized algorithms \cite{HMT11}.

Once the DEIM basis is fixed, there are several ways to select the DEIM interpolation indices so that $P^\top\deimb$ is non-singular. In this work we focus on the greedy algorithm of \cite[Algorithm 1]{CS10}. An alternative choice is, for example, QDEIM \cite{DG16} which is based on a QR factorization with column pivoting of the matrix $\deimb$.

\begin{remark}\label{rmk:FOMROMJ}
In principle, one may approximate the full Jacobian $J_{\prm,\G}$ rather than the reduced Jacobian $J_{\prm,\G}(\redb\cdot)\redb$. However, there are (at least) two major issues in dealing with the full Jacobian.
First, the snapshot matrix $M_J$ has $\Nf \ns\np$ columns instead of $\Nr \ns\np$, where $\nf\gg\nr$. In many applications, already for one-dimensional problems, this matrix is too large to be stored, and the cost of computing its singular values and vectors is prohibitive.
This is particularly relevant in the adaptive approach, see \Cref{sec:adaptive_DEIM}, when the DEIM projection is updated online.
Second, the full Jacobian $J_{\prm,\G}$ usually exhibits diagonal (or sparsity) patterns so that the associated snapshot matrix has (almost) full rank. This implies that one has to select $\nd\approx\nh$ DEIM basis functions to achieve a sufficiently accurate approximation, making hyper-reduction basically ineffective. On the contrary, when the reduced Jacobian is considered, it is found in practice that the singular values of $M_J$ exhibit a much faster decay, which allows for an efficient hyper-reduction. We refer to \Cref{fig:singvalues} in \Cref{sec:NLS} for a numerical example that shows this behavior.
\end{remark}

\subsection{\emph{A priori} convergence estimates}\label{sec:errestimate}
The error between the full order and hyper-reduced solution can be bounded by the error between the full order solution and its projection onto the reduced space and by the DEIM projection error of the reduced Jacobian.
The following theorem extends to our setting the result in \cite[Theorem 3.1]{CS12}.

\begin{theorem}\label{error}
For any $\prm\in\Sprm$,
let $\fs(\cdot,\prm)\in C^1(\Tcal,\Vd{\Nf})$ be the solution of the full model \eqref{eq:full_model}
and let
$\hrs(\cdot,\prm)\in C^1(\Tcal,\Vd{\Nr})$ be the solution of the hyper-reduced system \eqref{eq:Hyperred_model}.
Let $\Pbb$ be the DEIM projection \eqref{eq:DEIMproj}.
Assume that, for every $\prm\in\Sprm$, the Jacobian $J_{\prm,\G}$ of $\G$ \eqref{eq:decomp} is Lipschitz continuous in the $2$-norm with constant $\Lcal_{\G}(\prm)$. Then,
\begin{equation}\label{eq:error_bound_sol}
\begin{aligned}
	\norm{\fs-\redb\hrs}^2_{L^2(\Tcal\times\Sprm;\Vd{\Nf})} \leq 
	&\, C_1(T)\norm{\fs-\redb\redb^\top \fs}^2_{L^2(\Tcal\times\Sprm;\Vd{\Nf})}\\
    & + C_2(T)
	\norm{(I-\Pbb)J_{\prm,\G}(\redb\redb^{\top}\fs)\redb}^2_{L^2(\Tcal\times\Sprm;\Vd{\Nf})},
\end{aligned}
\end{equation}
where
$C_1(T):=\Delta\Tcal \max_{\prm\in\Sprm}\big(C_\prm(T)\alpha^2(\prm)\big)+1$
and
$C_2(T):=\Delta\Tcal\,\norm{\cv}^2 \max_{\prm\in\Sprm} C_\prm(T)$,
\begin{equation*}
    C_{\prm}(T):=
    \begin{cases}
    \beta^{-1}(\prm)\big(e^{2\beta(\prm)(T-t_0)}-1\big), &\quad \mbox{if}\;\beta(\prm)\neq 0, \\
    2(T-t_0), &\quad \mbox{if}\;\beta(\prm) =0,
    \end{cases}
\end{equation*}
with $\alpha(\prm) := \norm{\redb^\top K(\prm)}_2 + \Lcal_{\hn}(\prm)$, $\Lcal_{\hn}(\prm)$ Lipschitz continuity constant of $\nabla\hn(\cdot,\prm)$,
$\beta(\prm):=\mu(K(\prm)) + \Lcal_{\G}(\prm)\norm{(P^\top \deimb)^{-1}}_2\norm{\cv}$, and
$\mu(K(\prm))$ logarithmic norm of $K(\prm):=J_{\Nf}L(\prm)$ with respect to the 2-norm, i.e., $\mu(K) := \sup_{x\neq 0}\frac{(x,Kx)}{(x,x)}$.
\end{theorem}
\begin{proof}
Let us fix the parameter $\prm\in\Sprm$.
The error at each $t\in\Tcal$ can be split as
\begin{equation*}
    e(t,\prm) := \fs(t,\prm)-\redb\hrs(t,\prm)=e_p(t,\prm)+e_h(t,\prm),
\end{equation*}
where $e_p:=\fs-\redb\redb^{\top}\fs$ is the projection error and $e_h:=\redb\redb^{\top}\fs-\redb\hrs$.
Differentiating $e_h$ with respect to time, using \eqref{eq:full_model} and \eqref{eq:Hyperred_model}, and the symplecticity of $\redb$ results in
\begin{equation*}
    \Dot{e}_h(t,\prm)  = \redb\redb^{\top}\Dot{\fs}(t,\prm)-\redb\Dot{\hrs}(t,\prm)
     = \redb\redb^{\top} K(\prm) e_h(t,\prm) + s(t,\prm),
\end{equation*}
where $s(t,\prm) := \redb\redb^{\top}K(\prm)e_p(t,\prm)+\redb J_{\Nr}\big(\redb^{\top}\nabla_\fs \hn(\fs,\prm)-\nabla_\hrs \hn_{hr}(\hrs,\prm)\big)$ and $K(\prm):=J_{\Nf}L(\prm)$.
The norm of $s$ can be bounded as
\begin{equation}\label{eq:proof_2}
    \norm{s(t,\prm)} \leq
    \norm{\redb^\top K(\prm)}_2\norm{e_p(t,\prm)}
    + \norm{\redb^{\top}\nabla_\fs \hn(\fs,\prm)-\nabla_\hrs \hn_{hr}(\hrs,\prm)}.
\end{equation}
The second term in \eqref{eq:proof_2} can be split as follows,
\begin{align}\label{eq:3_terms_ineq}
    \norm{\redb^{\top}\nabla_\fs \hn(\fs,\prm)-\nabla_\hrs \hn_{hr}(\hrs,\prm)}\leq\, &
    \norm{\redb^{\top}\nabla_\fs \hn(\fs,\prm)-\redb^{\top}\nabla_\fs \hn(\redb\redb^{\top}\fs,\prm)}\\ \nonumber
    & + \norm{\redb^{\top}\nabla_\fs \hn(\redb\redb^{\top}\fs,\prm)-\nabla_\hrs \hn_{hr}(\redb^{\top}\fs,\prm)} \\\nonumber
    & + \norm{\nabla_\hrs \hn_{hr}(\redb^{\top}\fs,\prm)-\nabla_\hrs \hn_{hr}(\hrs,\prm)}.
\end{align}
We consider the three terms on the right-hand side separately.
First observe that, if $\nabla_y\Ham(\cdot,\prm)$ is Lipschitz continuous in the $2$-norm with constant $\Lcal_{\Ham}(\prm)$, then $\nabla\hn(\cdot,\prm)$ is also Lipschitz continuous with constant $\Lcal_{\hn}(\prm)\leq \Lcal_{\Ham}(\prm) + \norm{L(\prm)}_2$. This gives
$$\norm{\redb^{\top}\nabla_\fs \hn(\fs,\prm)-\redb^{\top}\nabla_\fs \hn(\redb\redb^{\top}\fs,\prm)}\leq L_{\hn}(\prm)\norm{e_p(t,\prm)}.$$
Moreover, using the decomposition of the Hamiltonian in \eqref{eq:decomp} results in
\begin{align*}
\begin{split}
        \norm{\redb^{\top}\nabla_\fs \hn(\redb\redb^{\top}\fs,\prm)-\nabla_\hrs \hn_{hr}(\redb^{\top}\fs,\prm)}
        & = \norm{\redb^{\top}J_{\prm,\G}^{\top}(\redb\redb^{\top}\fs)\cv-\redb^{\top}J_{\prm,\G}^{\top}(\redb\redb^{\top}\fs)\Pbb^{\top}\cv} \\
        & \leq\norm{\cv}\norm{(I-\Pbb)J_{\prm,\G}(\redb\redb^{\top}\fs)\redb}_2,
\end{split}
\end{align*}
and similarly
\begin{equation}\label{eq:proof_3}
\begin{aligned}
        \norm{\nabla_\hrs \hn_{hr}(\redb^{\top}\fs,\prm)-\nabla_\hrs \hn_{hr}(\hrs,\prm)}
        & =\norm{\redb^{\top}J_{\prm,\G}^{\top}(\redb\redb^{\top}\fs)\Pbb^{\top}\cv - \redb^{\top}J_{\prm,\G}^{\top}(\redb\hrs)\Pbb^{\top}\cv}\\
        &\leq L_\G(\prm)\norm{\redb\redb^\top\fs-\redb\hrs}\norm{\Pbb^\top \cv}\leq\gamma(\prm)\norm{e_h(t,\prm)},
    \end{aligned}
\end{equation}
    where $\gamma(\prm):=L_\G(\prm)\norm{(P^\top \deimb)^{-1}}_2\norm{\cv}$.
We can now substitute these three bounds into \eqref{eq:3_terms_ineq} and \eqref{eq:proof_2} to obtain
\begin{equation}\label{eq:alpha_beta_gamma_bound}
    \norm{s(t,\prm)} \leq \alpha(\prm)\norm{e_p(t,\prm)} + \gamma(\prm)\norm{e_h(t,\prm)} + \norm{\cv}\norm{W(t,\prm)}_2,
\end{equation}
where $\alpha(\prm) := \norm{\redb^\top K(\prm)}_2 + L_{\hn}(\prm)$, and
$W(t,\prm) := (I-\Pbb)J_{\prm,\G}(\redb\redb^{\top}\fs(t,\prm))\redb$.

The time derivative of the error norm gives
\begin{equation}\label{eq:bound_ddt_theta_tilde}
    \frac{d}{dt}\norm{e_h(t,\prm)}=\frac{1}{\norm{e_h(t,\prm)}}\bigg((\redb\redb^{\top} K(\prm)e_h(t,\prm),e_h(t,\prm))+(s(t,\prm),e_h(t,\prm))\bigg).
\end{equation}
The first term on the right-hand side of \eqref{eq:bound_ddt_theta_tilde} can be bounded as
in \cite[Theorem 3.1]{CS12}, by
$\mu(K(\prm))\norm{e_h(t,\prm)}$,
using the fact that $(\redb\redb^{\top} K(\prm)e_h,e_h)=
(\redb\redb^{\top} K(\prm)\redb\widetilde{e}_h,\redb\widetilde{e}_h)
\leq \mu(\redb^\top K\redb)\norm{\widetilde{e}_h}^2$,
with $\widetilde{e}_h:=\redb^\top\fs-\hrs$.

Using
\eqref{eq:alpha_beta_gamma_bound} for the second term of
%
\eqref{eq:bound_ddt_theta_tilde} yields
\begin{equation*}
    \frac{d}{dt}\norm{e_h(t,\prm)}
    \leq \beta(\prm)\norm{e_h(t,\prm)} + b(t,\prm),
\end{equation*}
where $\beta(\prm):=\mu(K(\prm))+\gamma(\prm)$, and $b(t,\prm) := \alpha(\prm)\norm{e_p(t,\prm)}+\norm{\cv}\norm{W(t,\prm)}_2$.

A standard application of the Gronwall inequality \cite{Gro19} yields the conclusion.
\end{proof}
The dependence of the constants in the bound \eqref{eq:error_bound_sol} on the norm of the vector $\cv$ can be avoided by a suitable normalization of the decomposition \eqref{eq:decomp}.
Observe that the first term on the right-hand side of \eqref{eq:error_bound_sol} is controlled via the reduced space, while the second term can be minimized in the construction of the DEIM basis.
Indeed, the error of the approximation of the reduced Jacobian given by the DEIM projection \eqref{eq:DEIMproj} can be bounded analogously to \cite[Lemma 3.2]{CS10} as
    \begin{equation*}
        \norm{(I-\mathbb{P})J_{\prm_j,\G}(\redb\redb^{\top}\fs^{s_i}(\prm_j))\redb}_2
        \leq\norm{(P^{\top}\deimb)^{-1}}_2\sqrt{\sum_{\ell=m+1}^{\rank(M_J)}\sigma_\ell^2},
    \end{equation*}
for any $i=1,\ldots,\ns$, and $j=1,\ldots,\np$,
where $\{\sigma_\ell\}_{\ell}$ are the singular values of $M_{J}$.

For Jacobian matrices having a particular structure, the Lipschitz continuity assumption of $J_{\prm,\G}$ in \Cref{error} might not be needed.
An example often encountered in applications is when the $i$th entry of $\G(\fs,\prm)$ only depends on the $i$th pair of symplectic variables, namely on the $i$th and $(i+\nf)$th components of $\fs$, for any $\fs\in\Vd{\Nf}$.
In this case, $\nh=\nf$ and $J_{\prm,\G}\in\Rbb^{\nf\times\Nf}$ is formed of two $\nf\times\nf$ diagonal blocks.
This implies that $\nabla_\hrs \hn_{hr}=J_{\prm,\G}^\top\Pbb^\top \cv$ is of the form $D\nabla \hn$ where
\begin{equation*}
    D:=\diag((\Pbb^\top \cv)_1,\dots,(\Pbb^\top \cv)_\nf,(\Pbb^\top \cv)_1,\dots,(\Pbb^\top \cv)_\nf)\in\Rbb^{\Nf\times\Nf}.
\end{equation*}
Therefore, the bound in \eqref{eq:proof_3} becomes
\begin{align*}
\begin{split}
    \norm{\nabla_\hrs \hn_{hr}(\redb^\top\fs,\prm)-\nabla_\hrs \hn_{hr}(\hrs,\prm)}
    & =\norm{A^\top D\big(\nabla_\fs \hn(\redb\redb^\top\fs,\prm)-\nabla_\fs \hn(\redb\hrs,\prm)\big)}\\
    &\leq\norm{D}_2 \Lcal_{\hn}(\prm) \norm{\redb\redb^\top\fs-\redb\hrs},
\end{split}
\end{align*}
where 
$\norm{D}_2
\leq\norm{\Pbb^\top \cv}_2
    \leq\norm{(P^\top\deimb)^{-1}}_2\norm{\cv}$.
Hence, \eqref{eq:proof_3} holds with the constant $\gamma(\prm):=L_{\hn}(\prm)\norm{(P^\top \deimb)^{-1}}_2\norm{\cv}$ and no Lipschitz continuity of $J_{\prm,\G}$ is required.

\subsection{Conservation of the Hamiltonian}\label{sec:ham_cons}
In this section we assess the error in the conservation of the Hamiltonian introduced by the hyper-reduction. In the following result we show that the error between the Hamiltonian evaluated at the full model solution and at the hyper-reduced solution is bounded by the approximation error of $\hn$. This implies that the hyper-reduced model guarantees preservation of the Hamiltonian for sufficiently large DEIM bases, that is, when the accuracy of the hyper-reduced model reaches the accuracy of the reduced model and up to errors due to the time integration scheme. We refer to \Cref{sec:numexp} for several numerical tests that corroborate this result.

\begin{proposition}\label{prop:ham_cons}
Let $\prm\in\Sprm$ be fixed. Let $\fs^j(\prm)$ be an approximation of the solution $\fs(t^j,\prm)$ of the full order system \eqref{eq:full_model} at time $t^j$, with $j=1,\dots,\nt$, obtained with a user-defined numerical time integrator. Similarly, let $\hrs^j(\prm)$ be an approximation of the solution $\hrs(t^j,\prm)$ of the hyper-reduced system \eqref{eq:Hyperred_model} at time $t^j$.
Assume that $\fs^0(\prm)\in\col(\redb)$ and $\G(\redb\redb^{\top}\fs^0(\prm),\prm)\in\col(\deimb)$, where $\col$ denotes the column space. Then, the error $\Delta \Ham_j(\prm) := \seminorm{\Ham(\fs^j(\prm),\prm)-\Ham(\redb\hrs^j(\prm),\prm)}$
satisfies
\begin{equation*}
    \Delta \Ham_j(\prm)\leq\seminorm{\cv^{\top}(I-\Pbb) \G(\redb\hrs^j(\prm),\prm)} + \varepsilon_{\Ham}^{[t^0,t^j]} + \varepsilon_{\Ham_{hr}}^{[t^0,t^j]},
\end{equation*}
where $\varepsilon_\Ham^{[t^0,t^j]}:=\seminorm{\Ham(\fs^{j})-\Ham(\fs^0)}$ and $\varepsilon_{\Ham_{hr}}^{[t^0,t^j]}:=\seminorm{\Ham_{hr}(\hrs^{j})-\Ham_{hr}(\hrs^{0})}$ are the errors in the Hamiltonian conservation,
in the temporal interval $[t^0,t^j]$, associated with the chosen temporal integrator.
\end{proposition}
\begin{proof}
In this proof we omit the dependence of $\fs^j$, $\hrs^j$ and $\Ham$ on $\prm$. The Hamiltonian error at the generic time instant $t^j$ can be bounded as 
\begin{equation}\label{eq:ham_cons_bound}
\begin{aligned}
    \Delta \Ham_j(\prm)\leq &\, 
    \seminorm{\Ham(\fs^j)-\Ham(\fs^0)}
    +\seminorm{\Ham_{hr}(\hrs^j)-\Ham_{hr}(\hrs^0)}\\
    & +\seminorm{\Ham(\redb\hrs^0)-\Ham_{hr}(\hrs^0)}
    + \seminorm{\Ham(\redb\hrs^j)-\Ham_{hr}(\hrs^j)} + \Delta \Ham_0(\prm).
\end{aligned}
\end{equation}
The first two terms in \eqref{eq:ham_cons_bound}
measure the Hamiltonian conservation with respect to the initial value in the full and hyper-reduced systems. These quantities only depend on the time integration scheme.
The term $\seminorm{\Ham(\redb\hrs^0)-\Ham_{hr}(\hrs^0)}=\seminorm{\cv^{\top}(I-\Pbb)\G(\redb\hrs^0)}$ vanishes since $\G(\redb\hrs^0)\in\col(\deimb)$ by assumption.
Similarly, the term $\Delta \Ham_0(\prm)$ vanishes due to the injectivity of the linear map $\redb$ and the assumption $\fs^0\in\col(\redb)$.

The conclusion follows by the definition of the full and reduced Hamiltonian.
\end{proof}

The bound on the Hamiltonian error hinges on the assumptions that $\fs^0(\prm)\in\col(\redb)$ and $\G(\redb\redb^\top \fs^0,\prm)\in\col(\deimb)$ for any value of the parameter $\prm\in\Sprm$. These conditions can be enforced via a shifting of the state variable and of the operator $\G$.
Let us introduce the variable $\fs_s(t,\prm):=\fs(t,\prm)-\fs^0(\prm)$ and the shifted operator $\G_s(\fs(t,\prm),\prm):=\G(\fs(t,\prm),\prm)-\G(\fs^0(\prm),\prm)$, for any $t\in\Tcal$ and $\prm\in\Sprm$.
Substituting into the expression of the Hamiltonian \eqref{eq:Ham} gives
\begin{equation}\label{eq:shiftH}
    \Ham_s(\fs_s,\prm):= \Ham(\fs_s+\fs^0,\prm) = \dfrac12 \fs_s^\top L(\prm)\fs_s + \fs_s^\top f_s(\prm) + \Gcal_s(\prm)+ \cv^{\top}\G_s(\fs_s+\fs^0,\prm),
\end{equation}
where
$f_s(\prm):=f(\prm) + L(\prm)\fs^0$ and
$\Gcal_s(\prm)=\Gcal(\prm)+1/2(\fs^0)^\top L(\prm)\fs^0+(\fs^0)^\top f(\prm)+ \cv^\top \G(\fs^0,\prm)$.
Then, the shifted state $\fs_s$ satisfies an Hamiltonian system with Hamiltonian $\Ham_s$ as in \eqref{eq:shiftH} and initial condition $\fs_s(t_0,\prm) = 0$.
Since $\fs_s^0(\prm)=0$
and $\G_s(\redb\redb^\top\fs_s^0+\fs^0,\prm)= \G_s(\fs^0,\prm)=\G(\fs^0,\prm)-\G(\fs^0,\prm)=0$
for any $\prm\in\Sprm$,
the shifted system satisfies the hypotheses of \Cref{prop:ham_cons}.
The shifted system is the one we consider as full order model in all numerical tests of \Cref{sec:numexp}.

\section{Adaptive gradient-preserving hyper-reduction}\label{sec:adaptive_DEIM}
In many cases of interest, such as convection-dominated phenomena and conservative dynamics, it is known that the solution space, under variation of time and parameter, has poor global reducibility properties. An analogous property is observed when considering the space of Hamiltonian gradients, although a rigorous connection between the reducibility of the two spaces is not know. We refer to \Cref{sec:NLS} for numerical evidence of these considerations. 

To address this lack of global reducibility,
%
we derive an adaptive strategy for the hyper-reduction of \eqref{eq:full_model} where the DEIM projection is updated in time while preserving the gradient structure of the Hamiltonian vector field.
The proposed approach is inspired by the method introduced in \cite{Peher15,Peher20}.

For the sake of exposition, we assume that basis and interpolation points are adapted every $\delta$ time instants starting from $t^{\delta_0}$, where $\delta,\delta_0>1$ are fixed hyper-parameters.
Assuming that $\nt-\delta_0$ is a multiple of $\delta$, the number of adaptations is $\nad=\delta^{-1}(\nt-\delta_0)$ and the adaptations are performed at time instants $\{t^{\ell_j}\}_{j=1}^{\nad}$, where $\ell_j:=\delta_0+(j-1)\delta$. We define $\ell_0:=0$ and note that $t^{\ell_{\nad+1}}=t^{\nt}=T$.
In principle, $\nad$ does not need to be fixed \emph{a priori} and it may be determined during time evolution of the hyper-reduced system based on e.g. suitable error criteria. In the proposed adaptive hyper-reduction algorithm, the basis update is performed for each instance of the test parameter. Let us then fix the problem parameter $\prm\in\Sprm$.

Between two updates the local hyper-reduced system to be solved in the temporal interval $[t^{\ell_j},t^{\ell_{j+1}}]$ reads
\begin{equation}\label{eq:lochr}
\left\{\begin{aligned}
    & \Dot{\hrs}(t,\prm)=J_{\Nr}\nabla_\hrs \Ham_{hr}^j(\hrs(t,\prm),\prm), & t\in(t^{\ell_j},t^{\ell_{j+1}}],\\
    & \hrs(t^{\ell_j},\prm) = \hrs^{\ell_j}(\prm),
\end{aligned}\right.
\end{equation}
for $j=0,\dots,\nad$, where the local hyper-reduced Hamiltonian is given by
\begin{equation*}
    \Ham^j_{hr}(\hrs,\prm):=\frac{1}{2}\hrs^\top L_r(\prm)\hrs + + \hrs^\top f_r(\prm) + \Gcal(\prm) + \cv^\top\Pbb_j\G(\redb\hrs,\prm),
\end{equation*}
and $\Pbb_j:=\deimb_j(P^\top_j\deimb_j)^{-1}P_j^\top$ is the local DEIM projection.

Starting from a DEIM basis $\deimb_0\in\Rbb^{\nh\times\nd}$ and interpolation matrix $P_0\in\Rbb^{\nh\times\nd}$, the idea of the adaptive hyper-reduction is to perform a rank-$r$ update of the DEIM basis to adapt the DEIM pair $(\deimb_j,P_j)$ to $(\deimb_{j+1},P_{j+1})$ for $j=0,\dots,\nad-1$.
The initial DEIM basis $\deimb_0$ is constructed in the offline phase from snapshots of the full model solutions in the first $\delta_0$ time steps, namely
\begin{equation*}
    M_{J}=[J_{\prm,\G}(\redb\redb^\top \fs^1(\prm))\redb\,\ldots\,J_{\prm,\G}(\redb\redb^\top\fs^{\delta_0}(\prm))\redb]\in\Rbb^{\nh\times\Nr\delta_0}.
\end{equation*}
Observe that one might also construct the initial DEIM pair $(\deimb_0,P_0)$ from a set of training parameters.

Next, at time instant $t^{\ell_{j+1}}$, with $j=0,\dots,\nad-1$, the DEIM pair $(\deimb_j,P_j)$ is updated based on the snapshots of the nonlinear Hamiltonian in a past temporal window of size $w\in\mathbb{N}$, with $w<\delta_0$.
More in details, consider the snapshot matrix of the reduced Jacobian in the temporal window $[t^{\ell_{j+1}-w+1},t^{\ell_{j+1}}]$, namely let
\begin{equation*}
    F_j:=[J_{\prm,\G}(\redb\hrs^{\ell_{j+1}-w+1})\redb\,\ldots\,J_{\prm,\G}(\redb\hrs^{\ell_{j+1}})\redb]\in\Rbb^{\nh\times\overline{w}},
\end{equation*}
where $\overline{w}:=\Nr w$.
The residual of the DEIM approximation of the columns of $F_j$ is given by
$R_j=\deimb_jC_j-F_j\in\Rbb^{\nh\times\overline{w}}$,
where $C_j\in\Rbb^{\nd\times\overline{w}}$ is the DEIM coefficient matrix
$C_j:=(P_j^{\top}\deimb_j)^{-1}P_j^{\top}F_j$.
Analogously to \cite{Peher15}, we update the DEIM basis matrix $\deimb_j$ to $\deimb_{j+1}$ via a rank-$r$ correction, that is
\begin{equation}\label{eq:Uupdate}
    \deimb_{j+1}=\deimb_j+\bm{\alpha}_j\bbf_j^{\top}
\end{equation}
with $\bm{\alpha}_j\in\Rbb^{\nh\times r}$ and $\bbf_j\in\Rbb^{\nd\times r}$ of rank $r\in\mathbb{N}$, $r\leq\nd$.
The update $\albf_j\bbf_j^{\top}$ is defined as the rank-$r$ matrix that minimizes the Frobenius norm of the residual at the sampling points collected in the matrix
$S_j=[\mathbf{e}_{s_1^{(j)}}\,\ldots\,\mathbf{e}_{s_{\nd_s}^{(j)}}]\in\Rbb^{\nh\times \nd_s}$
where $s_1^{(j)},\dots,s_{\nd_s}^{(j)}\in\{1,\dots,\nh\}$, namely it minimizes
\begin{equation*}
    \norm{S_j^\top(\deimb_{j+1}C_j-F_j)}_F^2
    =\norm{S_j^{\top}R_j+S_j^{\top}\albf_j\bbf_j^{\top}C_j}_F^2.
\end{equation*}
With the change of variable
$\abf_j:=S_j^\top\albf_j\in\Rbb^{\nd_s\times r}$,
the update \eqref{eq:Uupdate} boils down to solving the minimization problem
\begin{equation}\label{eq:min_prob}
    (\abf_j,\bbf_j)=\argmin_{(\abf,\bbf)\in\Vcal_r(\nd_s,\nd)}\norm{S_j^{\top}R_j+\abf\bbf^{\top}C_j}_F^2,
\end{equation}
over the space of rank-$r$ matrices defined as
\begin{equation}\label{eq:Vrspace}
\Vcal_r(\nd_s,\nd):=\{(\abf,\bbf)\in\Rbb^{\nd_s\times r}\times\Rbb^{\nd\times r}:\;\rank(\abf)=\rank(\bbf) = r\}.
\end{equation}
The number $\nd_s$ of sampling points is taken such that $\nd_s\geq\nd$ and $\nd_s\ll\nh$. The first requirement ensures that the minimization problem \eqref{eq:min_prob} has a non-trivial solution, while the second condition is enforced to avoid working in the potentially high dimension $\nh$.
Notice that, since only $\nd_s$ of the $\nh$ rows of $S_j\abf_j\in\Rbb^{\nh\times r}$ are non-zero, the update \eqref{eq:Uupdate} only modifies $\nd_s$ rows of the DEIM basis $\deimb_j$, that is
\begin{equation}\label{eq:Scheck}
    \Check{S}_j^{\top}\deimb_{j+1}=\Check{S}_j^{\top}\deimb_j\quad\mbox{where}\quad \Check{S}_j:=[\mathbf{e}_{s_{\nd_s+1}^{(j)}}\ldots\,\mathbf{e}_{s_\nh^{(j)}}]\in\Rbb^{\nh\times (\nh-\nd_s)},
\end{equation}
and $\{s_{\nd_s+1}^{(j)},\dots,s_{\nh}^{(j)}\}:=\{1,\ldots,\nh\}\setminus \{s_1^{(j)},\dots,s_{\nd_s}^{(j)}\}$.

\subsection{Rank-\texorpdfstring{$r$}{r} update of the DEIM basis}\label{sec:DEIM_basis_update}
The minimization problem \eqref{eq:min_prob} has been already considered in \cite[Lemma 3.5]{Peher15} for rank-$1$ updates of the DEIM basis. We extend the result to the general case of rank $r>1$.
\begin{theorem}\label{theo:minprob_eig_equivalence}
Let $r\leq\nd$ and assume that $C_j\in\mathbb{R}^{\nd\times\overline{w}}$ has full row-rank $\nd$. Let $\Vcal_r(\nd_s,\nd)$ be defined as in \eqref{eq:Vrspace}.
Then, the solution of the minimization problem
\begin{equation}\label{eq:objective}
    \min_{(\abf,\bbf)\in\Vcal_r(\nd_s,\nd)}\norm{S_j^{\top}R_j+\abf\bbf^{\top}C_j}_F^2,
\end{equation}
is given by $\abf=[a_1 \, \ldots \, a_r]\in\Rbb^{\nd_s\times r}$ and $\bbf=[b_1\,\ldots\, b_r]\in\Rbb^{\nd\times r}$ where $b_1,\dots,b_r$ are the eigenvectors of the generalized eigenvalue problem
\begin{equation}\label{eq:gen_eval}
    C_j(S_j^{\top}R_j)^{\top}(S_j^{\top}R_j)C_j^{\top}v=\lambda C_jC_j^{\top}v,
\end{equation}
corresponding to the $r$ largest eigenvalues $\lambda_1\geq\dots\geq\lambda_r\geq 0$, and 
\begin{equation}\label{eq:ai_from_bi}
    a_i=-\frac{1}{\norm{C_j^{\top}b_i}^2}S_j^{\top}R_jC_j^{\top}b_i \qquad i=1,\dots,r.
\end{equation}
Moreover, the solution $(\abf,\bbf)\in\Vcal_r(\nd_s,\nd)$ of \eqref{eq:objective} satisfies
\begin{equation}\label{eq:update_error}
    \norm{S_j^{\top}R_j+\abf\bbf^{\top}C_j}_F^2=\norm{S_j^{\top}R_j}_F^2-\sum_{i=1}^r\lambda_i.
\end{equation}
\end{theorem}
\begin{proof}
Let $\Ocal_j$ denote the objective function of the minimization problem \eqref{eq:objective}, namely $\Ocal_j(\abf,\bbf):=\norm{S_j^\top R_j+\abf\bbf^\top C_j}_F^2$.
We recast problem \eqref{eq:objective} in the space
$$\widehat{\Vcal}_r(\nd_s,\nd):=\{(\abf,\bbf)\in\Rbb^{\nd_s\times r}\times\Rbb^{\nd\times r}:\;\rank(\abf)=\rank(\bbf) = r,\,\abf^\top\abf\;\mbox{diagonal}\}.$$
Assume that the pair $(\cbf,\dbf)\in\Vcal_r(\nd_s,\nd)$ minimizes \eqref{eq:objective}.
Let $\cbf=\fbf Z$ be the thin QR factorization of $\cbf$ without normalization, where $\fbf\in\Rbb^{\nd_s\times r}$ and $Z\in\Rbb^{r\times r}$ is upper triangular with ones on the main diagonal. Define $\gbf:=\dbf Z^{\top}\in\Rbb^{\nd\times r}$.
Then $(\fbf,\gbf)\in\widehat{\Vcal}_r(\nd_s,\nd)$ and $\fbf\gbf^\top C_j=\cbf\dbf^\top C_j$, which implies
\begin{equation*}
    \min_{(\abf,\bbf)\in\Vcal_r(\nd_s,\nd)}\norm{S_j^\top R_j+\abf\bbf^\top C_j}_F^2=\min_{(\abf,\bbf)\in\widehat{\Vcal}_r(\nd_s,\nd)}\norm{S_j^\top R_j+\abf\bbf^\top C_j}_F^2,
\end{equation*}
and we can restrict our search to the case where the vectors $a_i$ are orthogonal.

In $\widehat{\Vcal}_r$, the objective function can be expanded as
\begin{equation}\label{eq:objective_expanded}
    \Ocal_j(\abf,\bbf) = \norm{S_j^\top R_j}_F^2+\sum_{i=1}^r\norm{a_i}^2 \norm{C_j^\top b_i}^2+2\sum_{i=1}^r\text{tr}\left((S_j^\top R_j)^\top a_ib_i^\top C_j\right).
\end{equation}
Imposing the gradient of \eqref{eq:objective_expanded} to vanish leads to decoupled equations for each pair $(a_i,b_i)$.
In particular, for each $i=1,\ldots,r$, the resulting problem boils down to the rank-$1$ minimization of \cite[Lemma 3.5]{Peher15}. For a fixed index $i$, $a_i$ is given by \eqref{eq:ai_from_bi},
and $b_i$ is eigenvector of the generalized eigenvalue problem
\begin{equation*}
    C_j(S_j^\top R_j)^\top(S_j^\top R_j)C_j^\top v=\lambda C_jC_j^\top v.
\end{equation*}
The corresponding eigenvalue $\lambda_i\in\mathbb{R}$ can be written as
\begin{equation*}
    \lambda_i=\frac{\norm{S_j^\top R_jC_j^\top b_i}^2}{\norm{C_j^\top b_i}^2}.
\end{equation*}
So far we know that if $b_1,\dots,b_r$ are \emph{any} $r$ eigenvectors of \eqref{eq:gen_eval} and $a_1,\dots,a_r$ are obtained from $b_i$ using \eqref{eq:ai_from_bi}, then $(\abf,\bbf)\in\widehat{\Vcal}_r$ is a stationary point of \eqref{eq:objective}.
We now want to identify the global minimum.
Starting from \eqref{eq:objective_expanded} and inserting the expression for $a_i$, the second term on the right-hand side of \eqref{eq:objective_expanded} becomes
\begin{equation*}
    \sum_{i=1}^r\norm{a_i}^2\norm{C_j^{\top}b_i}^2=\sum_{i=1}^r\frac{\norm{S^{\top}RC_j^{\top}b_i}^2}{\norm{C_j^{\top}b_i}^2}=\sum_{i=1}^r\lambda_i,
\end{equation*}
while the third term of \eqref{eq:objective_expanded} is
\begin{equation*}
    \sum_{i=1}^r\text{tr}\left((S^{\top}R)^{\top}a_i b_i^{\top}C_j\right)
    = \sum_{i=1}^r-\frac{1}{\norm{C_j^{\top}b_i}^2}\text{tr}(b_i^{\top}C_j(S^{\top}R)^{\top}(S^{\top}R)C_j^{\top}b_i)
    = -\sum_{i=1}^r\lambda_i.
\end{equation*}
Hence, the objective function \eqref{eq:objective_expanded} in $\widehat{\Vcal}_r$ satisfies \eqref{eq:update_error}.
Since $\bbf$ is required to have column-rank $r$, the global minimum of \eqref{eq:objective} is obtained by choosing as $b_1,\dots,b_r$ the eigenvectors of \eqref{eq:gen_eval} corresponding to the $r$ largest eigenvalues.
\end{proof}

If the matrix $C_j$ does not have full row-rank, that is, $r_C:=\text{rank}(C_j)<m$, one can proceed as suggested in \cite[Lemma 3.4]{Peher15} by performing a rank-revealing QR decomposition of $C_j$, i.e., $C_j=Q_jZ_j$ where $Q_j\in\Rbb^{\nd\times r_C}$ has orthogonal columns and $Z_j\in\Rbb^{r_C\times\overline{w}}$. Introducing $\zbf=Q_j^\top\bbf$, \Cref{theo:minprob_eig_equivalence} gives the matrices $(\abf,\zbf)\in\Vcal_r(\nd_s,r_C)$ that minimize $\norm{S_j^{\top}R_j+\abf\zbf^{\top}Z_j}_F^2$, and $\bbf$ can be recovered as $\bbf=Q_j\zbf$.

The steps to update the DEIM basis are summarized in \Cref{alg:basis_adapt}.
\begin{algorithm}[H]
\caption{Adaptation of the DEIM projection}\label{alg:basis_adapt}
\begin{algorithmic}[1]
{\small
\Procedure{$(\deimb_{j+1},P_{j+1})=$ adaptBasis}{$\{\hrs^{\ell_{j+1}-w+1},\ldots,\hrs^{\ell_{j+1}}\}$, $\redb$, $\deimb_j$, $P_j$, $S_j$, $r$}
\State Compute the matrix $F_j$ at the interpolation and sampling points: $P_j^{\top}F_j$, $S_j^{\top}F_j$.
\State Compute the DEIM coefficients $C_j=(P_j^{\top}\deimb_j)^{-1}(P_j^{\top}F_j)$.
\State Compute the DEIM residual at the sampling points: $S_j^{\top}R_j=S_j^{\top}\deimb_jC_j-S_j^{\top}F_j$.
\State Solve the generalized eigenvalue problem \eqref{eq:gen_eval}.
\State Set $b_1,\dots,b_r$ as the eigenvectors corresponding to the $r$ largest eigenvalues.
\State Compute $a_1,\dots,a_r$ using \eqref{eq:ai_from_bi}.
\State Set $\abf_j=[a_1\,\ldots\,a_r]$ and $\bbf_j=[b_1\,\ldots\,b_r]$.
\State $S_j^{\top}\deimb_{j+1} \gets S_j^{\top}\deimb_j + \abf_j\bbf_j^{\top}$ and $\Check{S}_j^{\top}\deimb_{j+1}\gets \Check{S}_j^{\top}\deimb_j$.
\State Compute the DEIM indices $P_{j+1}$ associated with $\deimb_{j+1}$.
\EndProcedure}
\end{algorithmic}
\end{algorithm}
For each value of the parameter, the complexity of updating the DEIM basis, as described in \Cref{alg:basis_adapt}, is $O(\nd_s\nr w(s_1+s_2))+O(\nd^2\nd_s)+O(\nd_s\nd\nr w)$.
In greater detail, under the assumption that the Jacobian $J_{\prm,\G}$ has at most $s_1+s_2$ non-zero entries per row, the computation of $P_j^\top F_j=P_j^\top J_{\prm,\G}(\redb\hrs^i)\redb$ for all $w$ hyper-reduced solutions $\hrs^i$ in the temporal window has complexity $O((s_1+s_2)\nd\nr w)$. 
Analogously, 
$S_j^\top F_j\in\Rbb^{\nd_s\times\overline{w}}$ requires $O((s_1+s_2)\nd_s\nr w)$ operations.
Next, the matrix-matrix multiplications involved in the computation of $C_j\in\Rbb^{\nd\times\overline{w}}$ (line 3)
and of $S_j^{\top}R_j\in\Rbb^{\nd_s\times\overline{w}}$ (line 4)
require $O(\nd^3)+O(\nd^2\nr w)+O(\nd_s\nd\nr w)$ operations.
%
The solution of the generalized eigenvalue problem \eqref{eq:gen_eval} has arithmetic complexity $O(\nd_s\nd\nr w)+O(\nd^2\nd_s)+O(\nd^3)$.
Finally, the total complexity of computing the matrix $\abf\in\Rbb^{\nd_s\times r}$ via \eqref{eq:ai_from_bi} is $O(r\nd \nd_s+r\nd\nr w)$.
The update of the DEIM interpolation points (line 11) can be performed, for example, by means of the greedy algorithm described in \cite{CS10} with complexity is $O(\nd^3)+O(\nd^2\nh)$, \emph{cf.} \cite{DG16}. This cost is linear in $\nh$ because, at each iteration of \cite[Algorithm 1]{CS10}, the norm of the $\nh$-dimensional DEIM residual needs to be computed.
To reduce the computational burden of this step, one can update only the indices associated with the DEIM basis vectors that have undergone the largest rotations in the DEIM basis update from $\deimb_{j}$ to $\deimb_{j+1}$, as suggested in \cite[Section 4.1]{Peher15}.

\subsection{Sampling points update}\label{sec:samp_points_update}
A crucial aspect of \Cref{alg:basis_adapt} is the definition of the sampling matrix $S_j$.
To determine the best possible choice of sampling points, we analyze the reduction in the residual associated with the DEIM update and the distance of the DEIM space to the best approximation space at a given time.
To this end, we first show that the DEIM update minimizes the projection error of the residual at the sampling points onto the space spanned by the rows of $C_j$.
\begin{lemma}\label{theo:update_error}
Let $\overline{r}$ be the rank of the matrix $S_j^\top R_jC_j^\top\in\Rbb^{\nd_s\times\nd}$ and assume that the DEIM coefficient matrix $C_j\in\Rbb^{\nd\times\overline{w}}$ has full row-rank.
Let $\deimb_{j+1}$ be the rank-$\overline{r}$ update given by \Cref{alg:basis_adapt}. If $\Cbb_j=C_j^{\top}(C_jC_j^{\top})^{-1}C_j\in\Rbb^{\overline{w}\times\overline{w}}$, then
\begin{equation*}
    \normF{S_j^{\top}(\deimb_{j+1}C_j-F_j)}=\normF{S_j^{\top}R_j(I-\Cbb_j)}.
\end{equation*}
\end{lemma}
\begin{proof}
For all $i=1,\dots,\overline{r}$, the eigenpairs $(\lambda_i,b_i)$ satisfy the generalized eigenvalue problem \eqref{eq:gen_eval}, i.e.
$C_j(S_j^{\top}R_j)^{\top}(S_j^{\top}R_j)C_j^{\top}\bbf=C_j C_j^{\top}\bbf\Lambda$,
where $\Lambda=\text{diag}(\lambda_1,\dots,\lambda_{\overline{r}})$.
Since $C_j$ has full row-rank, the matrix $C_jC_j^{\top}$ is symmetric and positive definite, thus admitting the Cholesky decomposition $C_jC_j^{\top}=L_jL_j^{\top}$, where $L_j\in\Rbb^{\nd\times\nd}$ is lower triangular. We introduce the change of variables $\cbf=L_j^{\top}\bbf$, so that the eigenvalue problem becomes
$L_j^{-1}C_j(S_j^{\top}R_j)^{\top}(S_j^{\top}R_j)C_j^{\top}L_j^{-\top}\cbf=\cbf\Lambda$.
This means that the eigenvalues of the generalized problem \eqref{eq:gen_eval} coincide with the squared singular values of the matrix $M_j:=S_j^{\top}R_jC_j^{\top}L_j^{-\top}\in\Rbb^{\nd_s\times\nd}$. Then, equation \eqref{eq:update_error} can be written as
\begin{equation}\label{eq:update_error_sigma}
    \norm{S_j^{\top}R_j+\abf\bbf^{\top}C_j}_F^2=\norm{S_j^{\top}R_j}_F^2-\sum_{\ell=1}^{\overline{r}}\sigma_{\ell}^2,
\end{equation}
where $\sigma_1\geq\dots\geq\sigma_{\overline{r}}>0$ are the singular values of $M_j$.
Since $\overline{r}=\text{rank}(S_j^\top R_jC_j^\top L_j^{-\top})$,
\begin{equation*}
    \norm{S_j^{\top}R_j+\abf\bbf^{\top}C_j}_F^2=
    \norm{S_j^{\top}R_j}_F^2-\norm{S_j^{\top}R_jC_j^{\top}L_j^{-\top}}_F^2.
\end{equation*}
Moreover, by the cyclic property of the trace, it holds
\begin{equation*}
    \norm{S_j^{\top}R_j}_F^2 - \norm{S_j^{\top}R_jC_j^{\top}L_j^{-\top}}_F^2= 
    \text{tr}((I-C_j^{\top}(C_jC_j^{\top})^{-1}C_j)(S_j^{\top}R_j)^{\top}(S_j^{\top}R_j)).
\end{equation*}
Since $\Cbb_j=\Cbb_j^{\top}$ and $\Cbb_j^2=\Cbb_j$, we have $I-\Cbb_j=(I-\Cbb_j)(I-\Cbb_j)^{\top}$, and
\begin{equation*}
    \norm{S_j^{\top}R_j+\abf\bbf^{\top}C_j}_F^2 = \text{tr}((I-\Cbb_j)(I-\Cbb_j)^{\top}(S_j^{\top}R_j)^{\top}(S_j^{\top}R_j)),
\end{equation*}
which concludes the proof.
\end{proof}

The number of non-zero eigenvalues in the generalized eigenvalue problem \eqref{eq:gen_eval} coincides with the number $\overline{r}$ of non-zero singular values of $S_j^\top R_jC_j^\top L_j^{-\top}$.
Furthermore, owing to \eqref{eq:update_error_sigma}, including eigenvectors of \eqref{eq:gen_eval} corresponding to zero eigenvalues as columns of $\bbf$ will have no effect on the value of the objective function \eqref{eq:objective}. This suggests that the rank $r$ of the update should be chosen so that $r\leq\overline{r}$.

In the following result we study the distance between the DEIM space resulting from the update described in \Cref{alg:basis_adapt} and the best possible approximation space, namely the one spanned by the columns of $F_j$, for $j=0,\ldots,\nad$. The resulting bound gives us an indication on how the sampling matrix can be chosen to minimize the error.
Note that the result in the following theorem is analogous to \cite[Proposition 2]{Peher20}.
\begin{theorem}\label{theo:rhoj_squared}
Assume that 
$C_j\in\Rbb^{\nd\times\overline{w}}$ has full row-rank. Let $\deimb_{j+1}=\deimb_j+\albf_j\bebf_j^{\top}$ be the rank-$\overline{r}$ update of $\deimb_j$ obtained with \Cref{alg:basis_adapt}.
Let $\Ucal_{j+1}$ be the space spanned by the columns of $\deimb_{j+1}$, and let the columns of $F_j$ belong to the $\nd$-dimensional space $\bUcal_{j+1}$. Then, the distance between $\bUcal_{j+1}$ and $\Ucal_{j+1}$ is bounded as
    \begin{equation}\label{eq:distance}
        d(\bUcal_{j+1},\Ucal_{j+1}):=\norm{\overline{\deimb}_{j+1}-\deimb_{j+1}\deimb_{j+1}^{\top}\overline{\deimb}_{j+1}}_F^2\leq\frac{\rho_j^2}{\sigma_{\min}^2(F_j)}
    \end{equation}
where $\sigma_{\min}(F_j)$ is the smallest nonzero singular value of $F_j$ and
    \begin{equation*}
        \rho_j^2:=\normF{R_j(I-\Cbb_j)}^2+\normF{\Check{S}_j^{\top}R_j\Cbb_j}^2=\normF{R_j}^2-\normF{S_j^{\top}R_j\Cbb_j}^2.
    \end{equation*}    
\end{theorem}
\begin{proof}
First note that the update $\albf_j\bebf_j^{\top}$ only changes the rows of $\deimb_j$ corresponding to the sampling points $S_j$ as shown in \eqref{eq:Scheck}.
\Cref{theo:update_error} yields
\begin{align}
\begin{split}\label{eq:update_error_expand}
    \normF{\deimb_{j+1}C_j-F_j}^2
    & =\normF{S_j^{\top}R_j(I-\Cbb_j)}^2+\normF{\Check{S}_j^{\top}R_j}^2\\
    & =\normF{R_j(I-\Cbb_j)}^2-\normF{\Check{S}_j^{\top}R_j(I-\Cbb_j)}^2+\normF{\Check{S}_j^{\top}R_j}^2.
\end{split}
\end{align}
The linearity and cyclic property of the trace, together with the fact that $\Cbb_j$ is a projection, give
$\normF{\Check{S}_j^{\top}R_j}^2-\normF{\Check{S}_j^{\top}R_j(I-\Cbb_j)}^2=\normF{\Check{S}_j^{\top}R_j\Cbb_j}^2$.
Hence, equation \eqref{eq:update_error_expand} yields $\normF{\deimb_{j+1}C_j-F_j}^2=\rho_j^2$.

Let the columns of $\overline{\deimb}_{j+1}$ form an orthonormal basis of $\bUcal_{j+1}$. Since the columns of $F_j$ belong to $\bUcal_{j+1}$, there is a matrix $\widetilde{F}_j\in\Rbb^{\nd\times\overline{w}}$ such that $F_j=\overline{\deimb}_{j+1}\widetilde{F}_j$. Hence,
\begin{align}
\begin{split}\label{eq:proj_error_Fj}
    \norm{F_j-\deimb_{j+1}\deimb_{j+1}^{\top}F_j}_F^2
    &=\norm{(\overline{\deimb}_{j+1}-\deimb_{j+1}\deimb_{j+1}^{\top}\overline{\deimb}_{j+1})\widetilde{F}_j}_F^2\\
    &\geq\norm{\overline{\deimb}_{j+1}-\deimb_{j+1}\deimb_{j+1}^{\top}\overline{\deimb}_{j+1}}_F^2\sigma_{\min}^2(\widetilde{F}_j),
\end{split}
\end{align}
and $\sigma_{\min}(\widetilde{F}_j)=\sigma_{\min}(F_j)$ since $\overline{\deimb}_{j+1}$ is orthonormal.
The conclusion follows by combining \eqref{eq:proj_error_Fj} with the bound
$\norm{F_j-\deimb_{j+1}\deimb_{j+1}^{\top}F_j}_F^2\leq\norm{F_j-\deimb_{j+1}C_j}_F^2.$
\end{proof}

If $\rho_j=0$ in \eqref{eq:distance}, then the spaces $\Ucal_{j+1}$ and $\bUcal_{j+1}$ coincide, and the adapted DEIM basis $\deimb_{j+1}$ can exactly represent all snapshots $F_j$ of the reduced Jacobian in the window.
This observation suggests that the ideal set of sample indices $S_j$ is the one that minimizes $\rho_j$. In particular, since $R_j$ and $\Cbb_j$ are fixed, we choose the $S_j$ that minimizes $\normF{\Check{S}_jR_j\Cbb_j}$ or, equivalently, maximizes $\normF{S_jR_j\Cbb_j}$, as described in \Cref{alg:indices_adapt}.
We remark the adaptive method proposed in \cite{Peher20}
is based on a different update of the DEIM basis, performed via SVD of the residual matrix $R_j$ (see \cite[Algorithm 3]{Peher20}), and the sampling matrix $S_j$ is chosen to minimize $\normF{\Check{S}_jR_j}$.
%
%
\begin{algorithm}[H]
\caption{Adaptation of the sampling indices}\label{alg:indices_adapt}
\begin{algorithmic}[1]
{\small
\Procedure{$S_{j+1}$=adaptSampleIndices}{$\{\hrs^{\ell_{j+1}-w+1},\ldots,\hrs^{\ell_{j+1}}\}$, $\redb$, $\deimb_j$, $P_j$, $\nd_s$}
\State Build the snapshot matrix $F_j$ of the reduced Jacobian in the window.
\State Compute $R_j\Cbb_j=\deimb_jC_j-F_jC_j^{\top}(C_jC_j^{\top})^{-1}C_j$.
\State Take as new sampling indices the indices of the $\nd_s$ rows of $R_j\Cbb_j$ with largest norm.
\EndProcedure}
\end{algorithmic}
\end{algorithm}

To analyze the computational cost of \Cref{alg:indices_adapt}, we first observe that some of the quantities involved are already available from the basis update (\Cref{alg:basis_adapt}).
The operations required are the computation of the $\nh-\nd$ entries of $F_j$ not available from $P_j^\top F_j$, at a cost of order $O((\nh-\nd)\nr w(s_1+s_2))$.
To construct $\Cbb_j$ one can use the SVD of $C_j=\Ucal_j\Sigma_j\Vcal_j^\top$ so that $\Cbb_j=\widetilde{\Vcal}_j\widetilde{\Vcal}_j^\top$, where $\widetilde{\Vcal}_j\in\Rbb^{\overline{w}\times\nd}$ is obtained by selecting the first $\nd$ columns of $\Vcal_j$.
In this case, $R_j\Cbb_j=U_jC_j-F_j\widetilde{\Vcal}_j\widetilde{\Vcal}_j^\top$ can be computed with $O(\nh\nd\nr w)$ operations.
The arithmetic complexity of \Cref{alg:indices_adapt} is then
linear in $\nh$.
Since $\nh$ typically scales with the full dimension $\nf$ (see discussion in \Cref{rmk:decomp}), the update of the sampling indices turns out to be too expensive to be performed at each adaptation step. In particular, solving the adaptive hyper-reduced system may become computationally demanding for large values of $w$ or $\nd$, as observed in the numerical experiments, \Cref{tab:w} and \Cref{fig:error_vs_ctime_2k400} in particular. For this reason,
it is preferable to perform the update of the sampling indices every $\gamma$ steps, where $\gamma=\nu\delta$ and $\nu>1$.
Moreover, instead of being fixed \emph{a priori}, the number of sampling indices $\nd_s$ may be determined at each adaptation until $\normF{\Check{S}_jR_j\Cbb_j}$ is below a threshold tolerance. We refer to \Cref{sec:NLS} for a numerical study regarding these aspects.

The gradient-preserving adaptive DEIM scheme is summarized in \Cref{alg:SP-ADEIM}.

\begin{algorithm}[H]
\caption{Gradient-preserving adaptive DEIM }\label{alg:SP-ADEIM}
\begin{algorithmic}[1]
{\small
\Procedure{GP-ADEIM}{$\redb$, $\delta_0$, $\delta$, $w$, $\gamma$, $\nd_s$, $r$}
\State Compute $(\deimb_0,P_0)$ from snapshots $J_{\prm,\G}(\redb\redb^\top\fs^{\ell}(\prm))$ for $\ell=0,\ldots,\delta_0$.
\State $j\gets0$
\For{$\tau=1,\ldots,\nt-1$}
\State Compute $\hrs^\tau$ by solving \eqref{eq:lochr} with $\deimb_j$ and $P_j$.
\If{mod($\tau-\delta_0,\delta$) $=0$}
\State Collect reduced trajectories in the window: $\Zcal=\{\hrs^{\delta_0+j\delta-w+1},\ldots,\hrs^{\delta_0+j\delta}\}$.
\If{mod($\tau-\delta_0,\gamma$) $=0$}
\State $S_{j}\gets$\textsc{adaptSampleIndices}(\Zcal, $\redb$, $\deimb_{j}$, $P_{j}$, $\nd_s$) with \Cref{alg:indices_adapt}.
\Else{}
\State $S_{j}\gets S_{j-1}$
\EndIf
\State $(\deimb_{j+1},P_{j+1})\gets$ \textsc{adaptBasis}($\Zcal$, $\redb$, $\deimb_{j}$, $P_{j}$, $S_{j}$, $r$) with \Cref{alg:basis_adapt}.
\State $j\gets (\tau-\delta_0)/\delta+1$
\EndIf
\EndFor
\EndProcedure}
\end{algorithmic}
\end{algorithm}

\subsection{Conservation of the Hamiltonian}\label{sec:hamcons_adapt}
Similarly to the study conducted in \Cref{sec:ham_cons} for the non-adaptive hyper-reduction,
we assess the error in the conservation of the Hamiltonian due to its approximation via \emph{local} hyper-reduction. In the following result we show that the error between the
Hamiltonian evaluated at the full model solution and at the hyper-reduced solution at a given time is bounded by the sum of the local hyper-reduction error of $\G$ at all previous updates.
\begin{proposition}\label{prop:ham_cons_adapt}
Let $\prm\in\Sprm$ be fixed. Let $\fs^{\ell_j}(\prm)$ be an approximation of the solution $\fs(t^{\ell_j},\prm)$ of the full order system \eqref{eq:full_model} at time $t^{\ell_j}$, with $j=1,\dots,\nad+1$, obtained with a user-defined numerical time integrator. Similarly, let $\hrs^{\ell_j}(\prm)$ be an approximation of the solution $\hrs(t^{\ell_j},\prm)$ of the hyper-reduced system \eqref{eq:Hyperred_model} at time $t^{\ell_j}$.
Assume that $\fs^0(\prm)\in\col(\redb)$ and $\G(\redb\redb^{\top}\fs^0(\prm),\prm)\in\col(\deimb_0)$. Furthermore, assume that the basis is updated so that $\G(\redb\hrs^{\ell_j}(\prm),\prm)\in\col(\deimb_{j})$ for all $j=1,\dots,\nad$.
Then, the error $\Delta \Ham_j(\prm) := \seminorm{\Ham(y^{\ell_j}(\prm),\prm)-\Ham(\redb\hrs^{\ell_j}(\prm),\prm)}$
satisfies
\begin{equation}\label{eq:Hambound}
    \Delta \Ham_j(\prm)\leq\sum_{i=0}^{j-1}\big(\seminorm{\cv^{\top}(I-\Pbb_i) \G(\redb\hrs^{\ell_{i+1}}(\prm),\prm)} +
    \varepsilon_{\Ham^i_{hr}}^{[t^{\ell_i},t^{\ell_{i+1}}]}\big)
    +\varepsilon_{\Ham}^{[t^0,t^{\ell_j}]},
\end{equation}
where $\varepsilon_{\Ham^j_{hr}}^{[t^{\ell_j},t^{\ell{j+1}}]}:=\seminorm{\Ham_{hr}^j(\hrs^{\ell_{j+1}})-\Ham_{hr}^j(\hrs^{\ell_j})}$ and $\varepsilon_{\Ham}^{[t^0,t^{\ell_j}]}:=\seminorm{\Ham(\fs^{\ell_j})-\Ham(\fs^0)}$ are the errors in the Hamiltonian conservation in the specified interval and associated with the chosen temporal integrator.
\end{proposition}
\begin{proof}
In this proof we omit the dependence of $\fs^{\ell_j}$, $\hrs^{\ell_j}$, and of $\Ham$ on $\prm$. The Hamiltonian error at the generic time instant $t^{\ell_j}$ can be bounded as 
\begin{equation*}
    \Delta \Ham_j(\prm)\leq\seminorm{\Ham(\fs^{\ell_j})-\Ham(\fs^0)}+\seminorm{\Ham(\fs^0)-\Ham(\redb\hrs^0)}+\seminorm{\Ham(\redb\hrs^0)-\Ham(\redb\hrs^{\ell_j})}.
\end{equation*}
The first term on the right-hand side is $\varepsilon_{\Ham}^{[t^0,t^{\ell_j}]}$ and only depends on the time integration scheme. Moreover, the second term is zero since $\fs^0\in\col(\redb)$ by assumption. As shown in \Cref{sec:ham_cons}, this assumption can be easily enforced by a suitable shift of the initial condition. The third term can be bounded as
\begin{align*}
\begin{split}
    \seminorm{\Ham(\redb\hrs^0)-\Ham(\redb\hrs^{\ell_j})}
    &\leq\seminorm{\Ham(\redb\hrs^0)-\Ham_{hr}^0(\hrs^0)}
    +\sum_{i=0}^{j-1}\seminorm{\Ham_{hr}^i(\hrs^{\ell_{i+1}})-\Ham_{hr}^i(\hrs^{\ell_{i}})} \\
    &+\sum_{i=0}^{j-2}\seminorm{\Ham_{hr}^{i+1}(\hrs^{\ell_{i+1}})-\Ham_{hr}^{i}(\hrs^{\ell_{i+1}})}
    +\seminorm{\Ham_{hr}^{j-1}(\hrs^{\ell_j})-\Ham(\redb\hrs^{\ell_j})}.
\end{split}
\end{align*}
The first term on the right-hand side vanishes if $\G(\redb\hrs^0)\in\col(\deimb_0)$. Furthermore, the quantity $\seminorm{\Ham_{hr}^i(\hrs^{\ell_i})-\Ham_{hr}^i(\hrs^{\ell_{i+1}})}=\varepsilon_{\Ham^i_{hr}}^{[t^{\ell_i},t^{\ell_{i+1}}]}$ only depends on the temporal integrator. The third term can be further split as follows, for all $i=0,\dots,j-2$,
\begin{equation*}
    \seminorm{\Ham_{hr}^i(\hrs^{\ell_{i+1}})-\Ham_{hr}^{i+1}(\hrs^{\ell_{i+1}})}\leq\seminorm{\Ham_{hr}^i(\hrs^{\ell_{i+1}})-\Ham(\redb\hrs^{\ell_{i+1}})}+\seminorm{\Ham(\redb\hrs^{\ell_{i+1}})-\Ham_{hr}^{i+1}(\hrs^{\ell_{i+1}})},
\end{equation*}
where the second term vanishes if $\G(\redb\hrs^{\ell_{i+1}})\in\col(\deimb_{i+1})$.

The conclusion follows from the definition of the Hamiltonians involved.
\end{proof}


\section{Numerical experiments}\label{sec:numexp}
To numerically assess the performances of the proposed methods, we consider two test problems: the 2D shallow water equations and the 1D nonlinear Schr\"odinger equation. As numerical time integration schemes we compare the implicit midpoint rule (IMR), which is a symplectic time integrator \cite{HLW06}, and the Average Vector Field (AVF), which is not symplectic but exactly preserves the Hamiltonian \cite{QMcL08,CELL12}. The tolerance of the Newton method used within the implicit timestepping is set to $10^{-10}$.
In all the numerical tests, we evaluate the accuracy of the reduced and hyper-reduced solutions by means of the following relative errors:
\begin{equation}\label{eq:errNR}
    \Ecal_{L^2}=\sqrt{\frac{\sum_{i=0}^{\nt}\norm{\fs^i-\redb\rs^i}^2}{\sum_{i=0}^{\nt}\norm{\fs^i}^2}}, \qquad \Ecal_{\fin}=\frac{\norm{\fs^{\nt}-\redb\rs^{\nt}}}{\norm{\fs^{\nt}}}.
\end{equation}
We shall distinguish between $\Ecal_{L^2}^{\R}$, $\Ecal_{\fin}^{\R}$ and $\Ecal_{L^2}^{\HR}$, $\Ecal_{\fin}^{\HR}$ according to whether $\rs^i$ is the reduced or the hyper-reduced solution at time $t^i$, respectively.

\subsection{Two-dimensional shallow water equations}
The two-dimensional shallow water equations (2D-SWE) \cite{RS13} in the rectangular domain $\Omega=[-L_{x_1},L_{x_1}]\times[-L_{x_2},L_{x_2}]\subset\mathbb{R}^2$ reads
\begin{equation}\label{eq:2Dsw_eq_final}
    \left\{\begin{aligned}
    & \partial_t\chi+\gamma\nabla\cdot(\chi\nabla\Phi), & \quad \text{in }\Omega\times(0,T], \\ & \partial_t\Phi+\frac{\gamma}{2}\lvert\nabla\Phi\rvert^2+\gamma\chi=0, & \quad \text{in }\Omega\times(0,T],
    \end{aligned}\right.
\end{equation}
where $\partial_t$ is short-hand notation for the temporal derivative, $\nabla\cdot$ and $\nabla$ denote the divergence and gradient operators in $\mathbf{x}=(x_1,x_2)\in\Omega$, respectively; $\chi=\chi(\mathbf{x},t;\prm)$ is the height of the free surface and $\Phi=\Phi(\mathbf{x},t;\prm)$ is the scalar velocity potential, that is, $\nabla\Phi$ is the horizontal velocity normalized by the characteristic speed; $\gamma\in\mathbb{R}$ is a constant-valued parameter.
We complement the problem with periodic boundary conditions and initial conditions.
Problem \eqref{eq:2Dsw_eq_final} admits a Hamiltonian formulation with Hamiltonian
\begin{equation*}
    \mathcal{\widehat{H}}(\chi,\phi)=\displaystyle\frac{\gamma}{2}\int_\Omega (\chi\lvert\nabla\phi\rvert^2+\chi^2)\,\dx.
\end{equation*}

We discretize the spatial differential operators in \eqref{eq:2Dsw_eq_final} using centered second-order finite differences. To this end, we introduce a uniform Cartesian mesh on the domain $\Omega$ with $((x_1)_i,(x_2)_j) = (-L_{x_1}+i\Delta x_1,-L_{x_2}+j\Delta x_2)$, for $i=0,\ldots,\nf_{x_1}$, $j=0,\ldots,\nf_{x_2}$ and $\Delta x_1=2L_{x_1}/\nf_{x_1}$, $\Delta x_2=2L_y/\nf_{x_2}$. Let $\nf:=\nf_{x_1}\nf_{x_2}$ and $D_{x_1},D_{x_2}\in\Rbb^{\nf\times\nf}$ be the matrices corresponding to the finite-difference discretization of the spatial derivatives. Defining the vectors $\bm{\chi}(t,\prm)=(\chi_1,\dots,\chi_\nf)^{\top}$, and $\bm{\Phi}(t,\prm)=(\Phi_1,\dots,\Phi_\nf)^{\top}$, approximating the nodal values of $\chi$ and $\Phi$ respectively, we obtain the discretized system 
\begin{equation*}
\begin{cases}
    \Dot{\bm{\chi}}+\gamma \left[D_{x_1}(\bm{\chi}\odot D_{x_1}\bm{\Phi})+D_{x_2}(\bm{\chi}\odot D_{x_2}\bm{\Phi})\right]=\mathbf{0}, \\[1ex]
    \Dot{\bm{\Phi}}+\gamma \bm{\chi}+\frac{\gamma}{2}\left[(D_{x_1}\bm{\Phi})^2+( D_{x_2}\bm{\Phi})^2\right]=\mathbf{0}.
\end{cases}
\end{equation*}
Here, the symbol $\odot$ denotes the pointwise product of vectors.
This is a Hamiltonian system of the form \eqref{eq:full_model} with $\fs(t,\prm)=(\bm{\chi}^{\top}(t,\prm),\bm{\Phi}^{\top}(t,\prm))^{\top}\in\Rbb^{\Nf}$ and
\begin{equation*}
    \hn(\fs,\prm)=\frac{\gamma}{2}\sum_{i=1}^\nf\chi_i\left[(D_{x_1}\bm{\Phi})^2_i+(D_{x_2}\bm{\Phi})^2_i\right]=\cv^{\top}\G(\fs,\prm),
\end{equation*}
where $\cv\in\Rbb^\nf$ is the vector having all entries equal to one and $\G(\fs,\prm)\in\Rbb^\nf$ is defined as $(\G(\fs,\prm))_i:=\frac{\gamma}{2}\chi_i\left[(D_{x_1}\bm{\Phi})^2_i+(D_{x_2}\bm{\Phi})^2_i\right]$ for all $i=1,\ldots,\nf$. Note that, in this case, the decomposition of $\hn$ has $\nh=\nf$.

In the following experiments we set $L_{x_1}=L_{x_2}=2$ and take $\nf_{x_1}=\nf_{x_2}=50$ grid points per direction, corresponding to $\Delta x_1=\Delta x_2=0.08$ and $\nf=2500$. The final time is $T=10$ and the time step is $\Delta t=0.005$. The initial condition is $\Phi(\mathbf{x},0)=0$, while the vertical displacement $\chi$ is a Gaussian profile of the form $\chi(\mathbf{x},0;\prm)=1+\frac12 e^{-\beta\lvert\mathbf{x}\rvert^2}$. The problem parameter is $\prm=(\beta,\gamma)$ and it ranges in the set $\Pcal=\left[1.1,1.7\right]\times\left[0.7,1.3\right]\subset\Rbb^2$.
The shift described in \Cref{sec:ham_cons} is performed on the initial condition.

For this test case, we compare the performances of the full model, the reduced model and the hyper-reduced model obtained with the non-adaptive algorithm. To construct the reduced model, the full system is solved, in the offline phase, for $5$ equispaced values of $\beta$ and $\gamma$, for a total of $25$ training parameters in the training set $\Sprmh$. Next, the reduced basis is built via complex SVD using all $\nt$ temporal snapshots for every training parameter.
We choose two different dimensions of the reduced space, namely $\Nr=80$ and $\Nr=160$.
For the hyper-reduced problem, we build the DEIM basis via POD on snapshots of the reduced Jacobian at every $20$ time steps for each training parameter.

We test the reduced model \eqref{eq:red_model} and the hyper-reduced model \eqref{eq:Hyperred_model} for the randomly selected parameter $\prm=(1.6435,0.7762)\notin\Sprmh$. In \Cref{fig:2dsw_errors_avf_imr} we report the relative errors \eqref{eq:errNR} of the hyper-reduced model vs. the size $\nd$ of the DEIM basis, obtained with the AVF and IMR time integrators and for the two sizes $\Nr\in\{80,160\}$ of the reduced space.
As expected, the error of the hyper-reduced model converges to the error of the reduced model as $\nd$ increases. Moreover, as the size $\Nr$ of the reduced model increases, a larger number of DEIM bases are required to approach the limit value. The errors obtained with the two numerical time integration schemes are very similar.

\begin{figure}[H]
\centering
\begin{tikzpicture}
    \begin{groupplot}[
      group style={group size=2 by 3,
                  horizontal sep=1.4cm},
      width=6.5cm, height=5cm
    ]
    \nextgroupplot[ylabel={},
                  xlabel={$\nd$},
                  axis line style = thick,
                  grid=both,
                  minor tick num=0,
                  max space between ticks=50,
                  grid style = {gray,opacity=0.2},
                  xmin=25, xmax=150,
                  ymax = 10^(-2), ymin = 10^(-4),
                  ymode=log,
                  xlabel style={font=\footnotesize},
                  ylabel style={font=\footnotesize},
                  x tick label style={font=\footnotesize},
                  y tick label style={font=\footnotesize},
                  legend style={font=\footnotesize},
                  legend columns = 4,
                  legend image post style={dash phase=0pt},
                  legend style={at={(1.1,1.35)},anchor=north}]
        \addplot+[color=blue,solid,mark=diamond] table[x=m,y=finHR] {2dsw_errors_k40_avf.txt};
        \addplot+[color=blue,dashed,dash pattern=on 6pt off 6pt,mark=none] table[x=m,y=finR] {2dsw_errors_k40_avf.txt};
        \addplot+[color=cyan,solid,mark=square] table[x=m,y=finHR] {2dsw_errors_k40_imr.txt};
        \addplot+[color=cyan,dashed,dash pattern=on 6pt off 6pt,dash phase=6pt,mark=none] table[x=m,y=finR] {2dsw_errors_k40_imr.txt};
        \addplot+[color=red,solid,mark=diamond] table[x=m,y=L2L2HR] {2dsw_errors_k40_avf.txt};
        \addplot+[color=red,dashed,dash pattern=on 6pt off 6pt,mark=none] table[x=m,y=L2L2R] {2dsw_errors_k40_avf.txt}; 
        \addplot+[color=orange,solid,mark=square] table[x=m,y=L2L2HR] {2dsw_errors_k40_imr.txt};
        \addplot+[color=orange,dashed,dash pattern=on 6pt off 6pt,dash phase=6pt,mark=none] table[x=m,y=L2L2R] {2dsw_errors_k40_imr.txt};
        \legend{{$\Ecal_{\fin}^{\HR}$, AVF$\quad$},
        {$\Ecal_{\fin}^{\R}$, AVF$\quad$},
        {$\Ecal_{\fin}^{\HR}$, IMR$\quad$},
        {$\Ecal_{\fin}^{\R}$, IMR},
        {$\Ecal_{L^2}^{\HR}$, AVF$\quad$},
        {$\Ecal_{L^2}^{\R}$, AVF$\quad$},
        {$\Ecal_{L^2}^{\HR}$, IMR$\quad$},
        {$\Ecal_{L^2}^{\R}$, IMR}};
    \nextgroupplot[ylabel={},
                  xlabel={$\nd$},
                  axis line style = thick,
                  grid=both,                   ytick={10^(-1),10^(-2),10^(-3),10^(-4),10^(-5),10^(-6)},
                  minor x tick num=1,
                  minor y tick num=1,
                  extra y ticks ={0.02,0.03,0.04,0.05,0.06,0.07,0.08,0.09,0.002,0.003,0.004,0.005,0.006,0.007,0.008,0.009,0.0002,0.0003,0.0004,0.0005,0.0006,0.0007,0.0008,0.0009,0.00002,0.00003,0.00004,0.00005,0.00006,0.00007,0.00008,0.00009,0.000002,0.000003,0.000004,0.000005,0.000006,0.000007,0.000008,0.000009},
                  extra y tick labels = {},
                  max space between ticks=50,
                  grid style = {gray,opacity=0.2},
                  xmin=50, xmax=300,
                  ymax = 10^(-2), ymin = 10^(-5),
                  ymode=log,
                  xlabel style={font=\footnotesize},
                  ylabel style={font=\footnotesize},
                  x tick label style={font=\footnotesize},
                  y tick label style={font=\footnotesize}]
        \addplot+[color=blue,solid,mark=diamond] table[x=m,y=finHR] {2dsw_errors_k80_avf.txt};
        \addplot+[color=blue,dashed,dash pattern=on 6pt off 6pt,mark=none] table[x=m,y=finR] {2dsw_errors_k80_avf.txt};
        \addplot+[color=cyan,solid,mark=square] table[x=m,y=finHR] {2dsw_errors_k80_imr.txt};
        \addplot+[color=cyan,dashed,dash pattern=on 6pt off 6pt,dash phase=6pt,mark=none] table[x=m,y=finR] {2dsw_errors_k80_imr.txt};
        \addplot+[color=red,solid,mark=diamond] table[x=m,y=L2L2HR] {2dsw_errors_k80_avf.txt};
        \addplot+[color=red,dashed,dash pattern=on 6pt off 6pt,mark=none] table[x=m,y=L2L2R] {2dsw_errors_k80_avf.txt};
        \addplot+[color=orange,solid,mark=square] table[x=m,y=L2L2HR] {2dsw_errors_k80_imr.txt};
        \addplot+[color=orange,dashed,dash pattern=on 6pt off 6pt,dash phase=6pt,mark=none] table[x=m,y=L2L2R] {2dsw_errors_k80_imr.txt};
    \end{groupplot}
\end{tikzpicture}
\caption{2D-SWE. Errors \eqref{eq:errNR} of the reduced and hyper-reduced models vs. size $\nd$ of the DEIM space. The size of the reduced model is $\Nr=80$ (left) and $\Nr=160$ (right). Comparison of AVF and IMR for time integration.}\label{fig:2dsw_errors_avf_imr}
\end{figure}
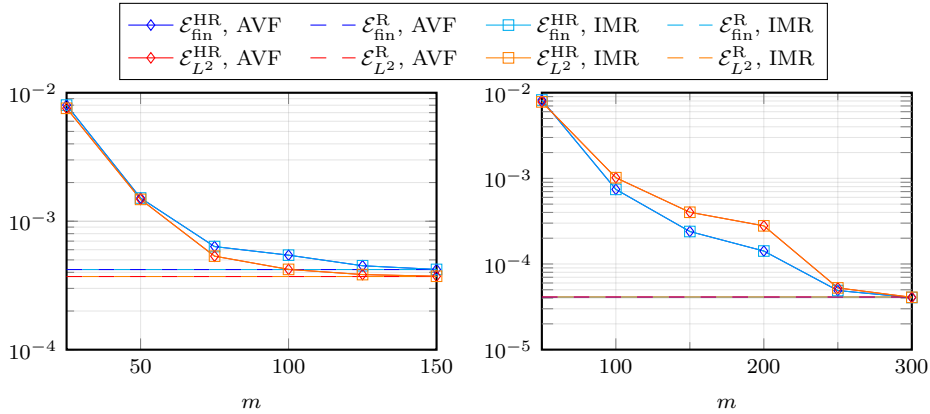

The computational times required by the reduced and hyper-reduced models solved using the AVF and the IMR time integrator are reported in \Cref{tab:2dsw_ctimes}. For $\Nr=80$, solving the hyper-reduced system with, e.g., $\nd=150$ is about $37$ times faster than solving the reduced system for AVF, and about $34$ times faster for IMR. For $\Nr=160$, solving the hyper-reduced system with $\nd=300$ is about $18$ times faster than solving the reduced system for AVF, and about $15$ times faster for IMR. Note that, in this test case, AVF is more computationally expensive than IMR, because an integral needs to be numerically evaluated at each time step.

\begin{table}[h!]
\footnotesize
\caption{2D-SWE. Computational times of the online phases of the reduced and hyper-reduced systems, $\Nr=80$ and $\Nr=160$. The computational time required to solve the full order system is $414.10\,s$ for AVF, $297.29\,s$ for IMR. The reported times are computed as averages over $3$ runs.}\label{tab:2dsw_ctimes}
\centering
\begin{tabular}{c|c|cccccc|}\cline{2-8}
& Reduced model & \multicolumn{6}{c|}{Hyper-reduced model}\\ \hline
\multicolumn{1}{|c|}{\multirow{4}{*}{AVF}} &
$\Nr=80$ & $\nd=25$ & $\nd=50$ & $\nd=75$ & $\nd=100$ & $\nd=125$ & $\nd=150$ \\ \cline{2-8}
\multicolumn{1}{|c|}{} & \textbf{276.02\,s} & \textbf{2.78\,s} & \textbf{3.86\,s} & \textbf{5.20\,s} & \textbf{5.81\,s} & \textbf{7.10\,s} & \textbf{7.44\,s} \\
\cline{2-8}\cline{2-8}\cline{2-8}
\multicolumn{1}{|c|}{} & $\Nr=160$ & $\nd=50$ & $\nd=100$ & $\nd=150$ & $\nd=200$ & $\nd=250$ & $\nd=300$ \\ \cline{2-8}
\multicolumn{1}{|c|}{} & \textbf{307.02\,s} & \textbf{7.44\,s} & \textbf{9.85\,s} & \textbf{12.16\,s} & \textbf{13.55\,s} & \textbf{16.08\,s} & \textbf{17.12\,s} \\ \hline\hline
\multicolumn{1}{|c|}{\multirow{4}{*}{IMR}} &
$\Nr=80$ & $\nd=25$ & $\nd=50$ & $\nd=75$ & $\nd=100$ & $\nd=125$ & $\nd=150$ \\ \cline{2-8}
\multicolumn{1}{|c|}{} & \textbf{128.22\,s} & \textbf{1.08\,s} & \textbf{1.60\,s} & \textbf{2.54\,s} & \textbf{2.96\,s} & \textbf{3.40\,s} & \textbf{3.71\,s} \\
\cline{2-8}\cline{2-8}\cline{2-8}
\multicolumn{1}{|c|}{} & $\Nr=160$ & $\nd=50$ & $\nd=100$ & $\nd=150$ & $\nd=200$ & $\nd=250$ & $\nd=300$ \\ \cline{2-8}
\multicolumn{1}{|c|}{} & \textbf{146.67\,s} & \textbf{5.25\,s} & \textbf{6.26\,s} & \textbf{7.19\,s} & \textbf{8.13\,s} & \textbf{9.36\,s} & \textbf{10.04\,s} \\ \hline
\end{tabular}
\end{table}

We then consider the conservation of the Hamiltonian for a fixed parameter $\prm=(1.6435,0.7762)$, whose dependence we omit in the following. In particular, we monitor the error $\seminorm{\Ham(\fs^0)-\Ham(\redb\hrs^j)}$ for all time indices $j=0,\ldots,\nt$.
This quantity is bounded, in turn, by the error $\Delta \Ham_j$ in approximating the reduced Hamiltonian with the hyper-reduced one: as derived in \Cref{sec:ham_cons}, it holds
\begin{equation*}
\seminorm{\Ham(\fs^0)-\Ham(\redb\hrs^j)} \leq \varepsilon_{\Ham}^{[t^0,t^j]} + \Delta \Ham_j\qquad\mbox{for all}\; j=0,\ldots,\nt.
\end{equation*}
\Cref{fig:2dsw_hamcons} shows the three terms involved in the bound at every time instant $t^j$ when the reduced space has dimension $\Nr=160$.
The error $\varepsilon_\Ham^{[t^0,t^j]}$ in the conservation of the full order Hamiltonian (black line in \Cref{fig:2dsw_hamcons}) only depends on the timestepping and reaches the tolerance of the Newton solver for AVF, while it depends on $\Delta t$ for IMR.
The error in the conservation of the Hamiltonian associated with the hyper-reduced model decreases as $\nd$ increases, that is, as the quality of the approximation of the Hamiltonian improves.
For $\nd$ sufficiently large, in this test case $\nd\geq 350$, the error due to the time integration is dominating over the error introduced by the Hamiltonian hyper-reduction.
The reason is that the hyper-reduced model with $\nd\geq 350$ is as accurate as the reduced model and, hence, the Hamiltonian is exactly preserved up to the error of the temporal integrator. As a further confirmation of the relationship between the conservation and the hyper-reduction of the Hamiltonian stated in \Cref{prop:ham_cons}, we also plot the DEIM approximation error $\seminorm{\Ham_{hr}(\hrs^j)-\Ham(\redb\hrs^j)}=\seminorm{\cv^\top(\Pbb \G(\redb\hrs^j)-\G(\redb\hrs^j))}$ for all $j=0,\ldots\nt$.

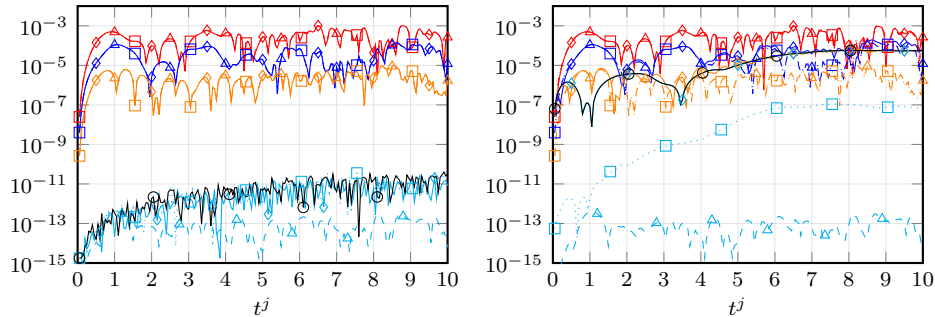
\begin{figure}[H]
\centering
\begin{tikzpicture}
    \begin{groupplot}[
      group style={group size=2 by 3,
                  horizontal sep=1.4cm},
      width=6.5cm, height=5cm
    ]
    \nextgroupplot[ylabel={},
                  xlabel={$t^j$},
                  xlabel style = {yshift=.2cm},
                  axis line style = thick,
                  grid=both,
                  minor tick num=0,
                  max space between ticks=50, ytick={10^(-15),10^(-13),10^(-11),10^(-9),10^(-7),10^(-5),10^(-3)},
                  xtick={0,1,2,3,4,5,6,7,8,9,10},
                  grid style = {gray,opacity=0.2},
                  xmin=0, xmax=10,
                  ymax = 10^(-2), ymin = 10^(-15),
                  ymode=log,
                  xlabel style={font=\footnotesize},
                  ylabel style={font=\footnotesize},
                  x tick label style={font=\footnotesize},
                  y tick label style={font=\footnotesize},
                  legend style={font=\footnotesize},
                  legend columns = 3,
                  legend style={at={(1.1,1.7)},anchor=north}]
        \addplot+[color=red,dotted,mark=square,mark repeat=30,mark phase=1,mark options={solid}] table[x=t,y=DeltaHim200] {2dsw_hamiltonian_avf_pt2.txt};
        \addplot+[color=red,solid,mark=diamond,mark repeat=30,mark phase=10,mark options={solid}] table[x=t,y=Hy0mHAzim200] {2dsw_hamiltonian_avf_pt2.txt};
        \addplot+[color=red,dashed,mark=triangle,mark repeat=30,mark phase=20,mark options={solid}] table[x=t,y=HhrzimHAzim200] {2dsw_hamiltonian_avf.txt};
        \addplot+[color=blue,dotted,mark=square,mark repeat=30,mark phase=1,mark options={solid}] table[x=t,y=DeltaHim250] {2dsw_hamiltonian_avf_pt2.txt};
        \addplot+[color=blue,solid,mark=diamond,mark repeat=30,mark phase=10,mark options={solid}] table[x=t,y=Hy0mHAzim250] {2dsw_hamiltonian_avf_pt2.txt};
        \addplot+[color=blue,dashed,mark=triangle,mark repeat=30,mark phase=20,mark options={solid}] table[x=t,y=HhrzimHAzim250] {2dsw_hamiltonian_avf.txt};
        \addplot+[color=orange,dotted,mark=square,mark repeat=30,mark phase=1,mark options={solid}] table[x=t,y=DeltaHim300] {2dsw_hamiltonian_avf_pt2.txt};
        \addplot+[color=orange,solid,mark=diamond,mark repeat=30,mark phase=10,mark options={solid}] table[x=t,y=Hy0mHAzim300] {2dsw_hamiltonian_avf_pt2.txt};
        \addplot+[color=orange,dashed,mark=triangle,mark repeat=30,mark phase=20,mark options={solid}] table[x=t,y=HhrzimHAzim300] {2dsw_hamiltonian_avf.txt};
        \addplot+[color=cyan,dotted,mark=square,mark repeat=30,mark phase=1,mark options={solid}] table[x=t,y=DeltaHim350] {2dsw_hamiltonian_avf_pt2.txt};
        \addplot+[color=cyan,solid,mark=diamond,mark repeat=30,mark phase=10,mark options={solid}] table[x=t,y=Hy0mHAzim350] {2dsw_hamiltonian_avf_pt2.txt};
        \addplot+[color=cyan,dashed,mark=triangle,mark repeat=30,mark phase=20,mark options={solid}] table[x=t,y=HhrzimHAzim350] {2dsw_hamiltonian_avf.txt};
        \addplot+[color=black,solid,mark=o,mark repeat=40] table[x=t,y=epsH] {2dsw_hamiltonian_avf_pt2.txt};
        \legend{{$\Delta \Ham_j, \nd=200\quad$},{$\seminorm{\Ham(\fs^0)-\Ham(\redb\hrs^j)}, \nd=200\quad$},{$\seminorm{\Ham_{hr}(\hrs^j)-\Ham(\redb\hrs^j)}$, $\nd=200$},{$\Delta \Ham_j, \nd=250\quad$},{$\seminorm{\Ham(\fs^0)-\Ham(\redb\hrs^j)}, \nd=250\quad$},{$\seminorm{\Ham_{hr}(\hrs^j)-\Ham(\redb\hrs^j)}$, $\nd=250$},{$\Delta \Ham_j, \nd=300\quad$},{$\seminorm{\Ham(\fs^0)-\Ham(\redb\hrs^j)}, \nd=300\quad$},{$\seminorm{\Ham_{hr}(\hrs^j)-\Ham(\redb\hrs^j)}$, $\nd=300$},{$\Delta \Ham_j, \nd=350\quad$},{$\seminorm{\Ham(\fs^0)-\Ham(\redb\hrs^j)}, \nd=350\quad$},{$\seminorm{\Ham_{hr}(\hrs^j)-\Ham(\redb\hrs^j)}$, $\nd=350$},{}};
    \nextgroupplot[ylabel={},
                  xlabel={$t^j$},
                  xlabel style = {yshift=.2cm},
                  axis line style = thick,
                  grid=both,
                  minor tick num=0,
                  max space between ticks=50, ytick={10^(-15),10^(-13),10^(-11),10^(-9),10^(-7),10^(-5),10^(-3)},
                  xtick={0,1,2,3,4,5,6,7,8,9,10},
                  grid style = {gray,opacity=0.2},
                  xmin=0, xmax=10,
                  ymax = 10^(-2), ymin = 10^(-15),
                  ymode=log,
                  xlabel style={font=\footnotesize},
                  ylabel style={font=\footnotesize},
                  x tick label style={font=\footnotesize},
                  y tick label style={font=\footnotesize}]
        \addplot+[color=red,dotted,mark=square,mark repeat=30,mark phase=1,mark options={solid}] table[x=t,y=DeltaHim200] {2dsw_hamiltonian_imr_pt2.txt};
        \addplot+[color=red,solid,mark=diamond,mark repeat=30,mark phase=10,mark options={solid}] table[x=t,y=Hy0mHAzim200] {2dsw_hamiltonian_imr_pt2.txt};
        \addplot+[color=red,dashed,mark=triangle,mark repeat=30,mark phase=20,mark options={solid}] table[x=t,y=HhrzimHAzim200] {2dsw_hamiltonian_imr.txt};
        \addplot+[color=blue,dotted,mark=square,mark repeat=30,mark phase=1,mark options={solid}] table[x=t,y=DeltaHim250] {2dsw_hamiltonian_imr_pt2.txt};
        \addplot+[color=blue,solid,mark=diamond,mark repeat=30,mark phase=10,mark options={solid}] table[x=t,y=Hy0mHAzim250] {2dsw_hamiltonian_imr_pt2.txt};
        \addplot+[color=blue,dashed,mark=triangle,mark repeat=30,mark phase=20,mark options={solid}] table[x=t,y=HhrzimHAzim250] {2dsw_hamiltonian_imr.txt};
        \addplot+[color=orange,dotted,mark=square,mark repeat=30,mark phase=1,mark options={solid}] table[x=t,y=DeltaHim300] {2dsw_hamiltonian_imr_pt2.txt};
        \addplot+[color=orange,solid,mark=diamond,mark repeat=30,mark phase=10,mark options={solid}] table[x=t,y=Hy0mHAzim300] {2dsw_hamiltonian_imr_pt2.txt};
        \addplot+[color=orange,dashed,mark=triangle,mark repeat=30,mark phase=20,mark options={solid}] table[x=t,y=HhrzimHAzim300] {2dsw_hamiltonian_imr.txt};
        \addplot+[color=cyan,dotted,mark=square,mark repeat=30,mark phase=1,mark options={solid}] table[x=t,y=DeltaHim350] {2dsw_hamiltonian_imr_pt2.txt};
        \addplot+[color=cyan,solid,mark=diamond,mark repeat=30,mark phase=10,mark options={solid}] table[x=t,y=Hy0mHAzim350] {2dsw_hamiltonian_imr_pt2.txt};
        \addplot+[color=cyan,dashed,mark=triangle,mark repeat=30,mark phase=20,mark options={solid}] table[x=t,y=HhrzimHAzim350] {2dsw_hamiltonian_imr.txt};
        \addplot+[color=black,solid,mark=o,mark repeat=40] table[x=t,y=epsH] {2dsw_hamiltonian_imr_pt2.txt};
    \end{groupplot}
\end{tikzpicture}
\caption{2D-SWE. Conservation of the Hamiltonian over time for different choices of $\nd$. The solid black lines with circles correspond to the error $\varepsilon_{\Ham}^{[t^0,t^j]}$ in the conservation of the full Hamiltonian due to the timestepping. AVF (left) and IMR (right) for time integration.
The reduced model has dimension $\Nr=160$ and the test parameter is $\prm=(1.6435,0.7762)$.}\label{fig:2dsw_hamcons}
\end{figure}

\subsection{Nonlinear Schr\"odinger equation}\label{sec:NLS}
As a second test case, we consider the nonlinear Schr\"odinger equation (NLS)
\begin{equation}\label{eq:schrodinger}
\imath\,\partial_tu + \partial_{xx}u+\epsilon\seminorm{u}^2u = 0,\quad x\in[-L,L], \,t\in(0,T],
\end{equation}
with initial condition $u(x,0;\prm)=\sqrt{2}(\cosh{x})^{-1}\exp{\left(\imath\frac{x}{2}\right)}$ and periodic boundary conditions. Here $\partial_{xx}$ denotes the second order derivative with respect to the spatial variable, and $\epsilon=\prm\in\Rbb$ is a parameter ranging in the set $\Pcal=[0.9,1.1]$. Writing $u(x,t;\prm)=q(x,t;\prm)+\imath p(x,t;\prm)$, for all $(x,t)\in[-L,L]\times[0,T]$, problem \eqref{eq:schrodinger} admits a Hamiltonian formulation with Hamiltonian given by
\begin{equation*}
    \mathcal{\widehat{H}}(q,p)=\frac{1}{2}\int_{-L}^L\left(\left(\partial_xp\right)^2+\left(\partial_xq\right)^2-\frac{\epsilon}{2}(q^2+p^2)^2\right)\,\dx.
\end{equation*}

Consider the uniform computational grid $x_i = -L+i\Delta x$, for $i=0,\ldots,\nf$ and
$\Delta x=2L/\nf$. Let $D_{xx}\in\Rbb^{\nf\times\nf}$ be the matrix corresponding to the central finite-difference discretization of the second order spatial derivative. Introducing the vectors $\mathbf{q}(t,\prm)=(q_1,\dots,q_\nf)^{\top}$ and $\mathbf{p}(t,\prm)=(p_1,\dots,p_\nf)^{\top}$,
where $q_i$ and $p_i$ are approximations to $q(x_i)$ and $p(x_i)$, we derive the system
\begin{equation*}
    \begin{cases}
    \dot{\mathbf{q}}=-D_{xx}\mathbf{p}-\epsilon(\mathbf{q}^2+\mathbf{p}^2)\odot\mathbf{p}, \\
    \dot{\mathbf{p}}=D_{xx}\mathbf{q}+\epsilon(\mathbf{q}^2+\mathbf{p}^2)\odot\mathbf{q},
    \end{cases}
\end{equation*}
which is of the form \eqref{eq:full_model} with $\fs(t,\prm)=(\mathbf{q}^\top(t,\prm),\mathbf{p}^\top(t,\prm))^\top\in\Rbb^{\Nf}$ and
\begin{equation*}
    \hn(\fs)=-\frac{\epsilon}{4}\sum_{i=1}^\nf(q_i^2+p_i^2)^2=\cv^{\top}\G(\fs).
\end{equation*}
Here $\cv\in\Rbb^\nf$ is the vector with entries all equal to one and $\G(\fs)\in\Rbb^\nf$ is defined as
$(\G(\fs))_i:=-\frac{\epsilon}{4}(q_i^2+p_i^2)^2$ for all $i=1,\ldots,\nf$.
In the following tests we set $L=\pi/l$ with $l=0.11$ as in \cite{AH17} and $\nf=2048$. AVF time integrator is applied in the temporal interval $[0,30]$ with time step $\Delta t=0.01$.

The reduced basis is built using the complex SVD on snapshots of the full model solution at each time step and for $\seminorm{\Sprmh}=11$ equispaced training parameters in the interval $\Pcal=[0.9,1.1]$. Each simulation of the full model takes about 220 seconds.

Before studying the performances of the proposed hyper-reduction technique, we compare it with the SDEIM proposed in \cite{AH17} and \cite{PM16}. SDEIM consists in applying a DEIM approximation to the nonlinear Hamiltonian gradient and to chose as DEIM basis the symplectic basis used to approximate the state, constructed by adding suitable snapshots to improve accuracy. The resulting method does not exactly preserve the gradient structure of the Hamiltonian vector field, as pointed out in \cite[page A17]{PM16} and in \cite[page 1717]{HP20}, but it ensures \emph{asymptotic} boundedness of the energy of the system \cite[Theorem 5.1]{PM16}. However, this property is not enough to guarantee stability of the resulting approximate solution.
To illustrate these facts, we compare our hyper-reduction technique (non-adaptive version) with SDEIM tested on the Schr\"odinger equation with $\prm = 1.0932$. Since in the SDEIM algorithm the DEIM basis is chosen to be the same as the reduced basis, we set $\nd=\Nr$ in the gradient-preserving hyper-reduction method so that the reduced and DEIM spaces have the same dimension for the two methods.
In particular we consider $\nr=\nd/2\in\{10,20,30,40,50,60\}$. To build the reduced and DEIM spaces, we consider all the snapshots of the full solution from the training phase, and one snapshot every $d_{snap}=20$ time steps for the nonlinear term.
\Cref{fig:pcolor_n40_n60} shows a pseudocolor plot of the numerical solutions on a space-time domain: The solution produced by SDEIM is unstable for $\Nr=80$ (first row) and inaccurate for $\Nr=120$ (second row). We also point out that, similarly, when classical DEIM is used to approximate the nonlinear term unstable solutions and unbounded errors are produced.

\begin{figure}[h!]
    \centering
	\includegraphics[width=\textwidth,height=5cm]{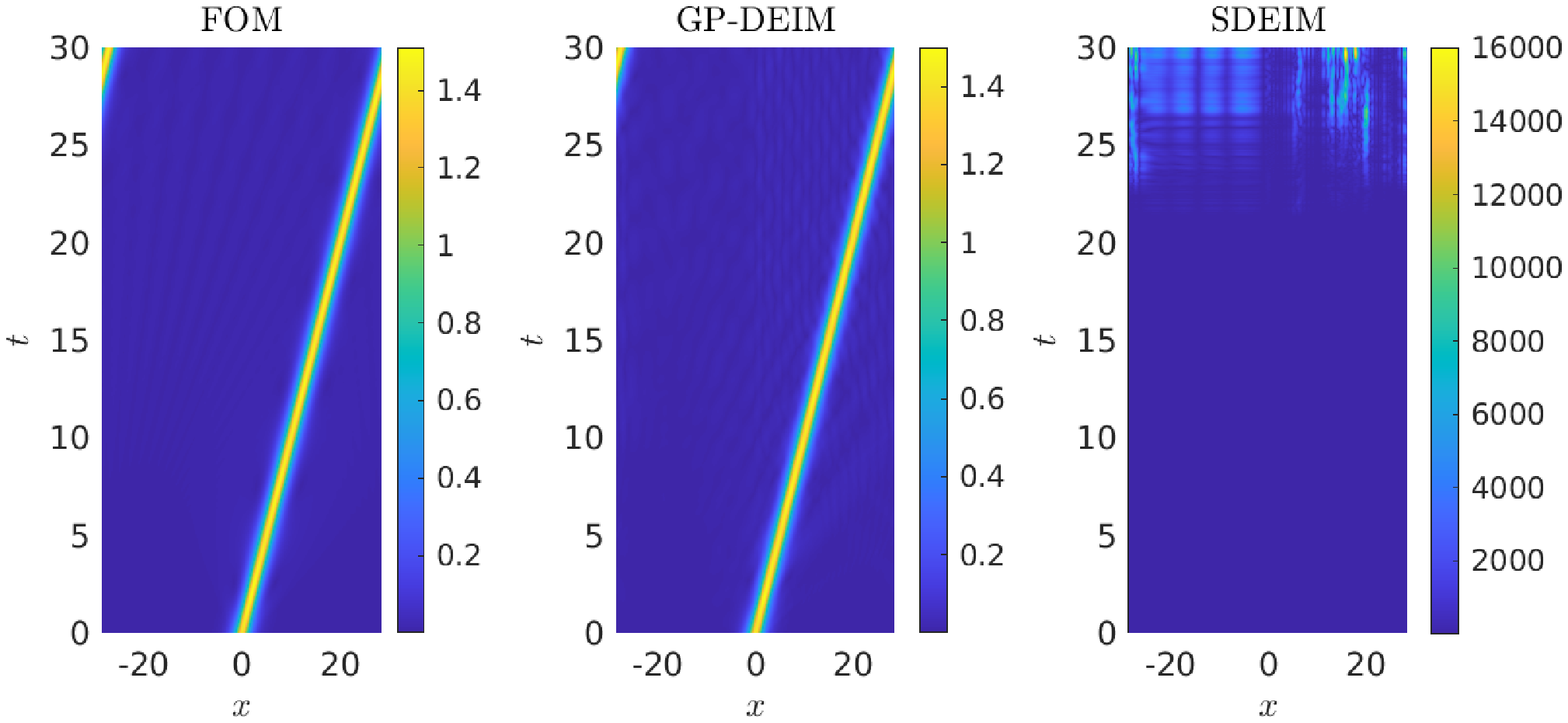}\\[2ex]
    \includegraphics[width=\textwidth,height=5cm]{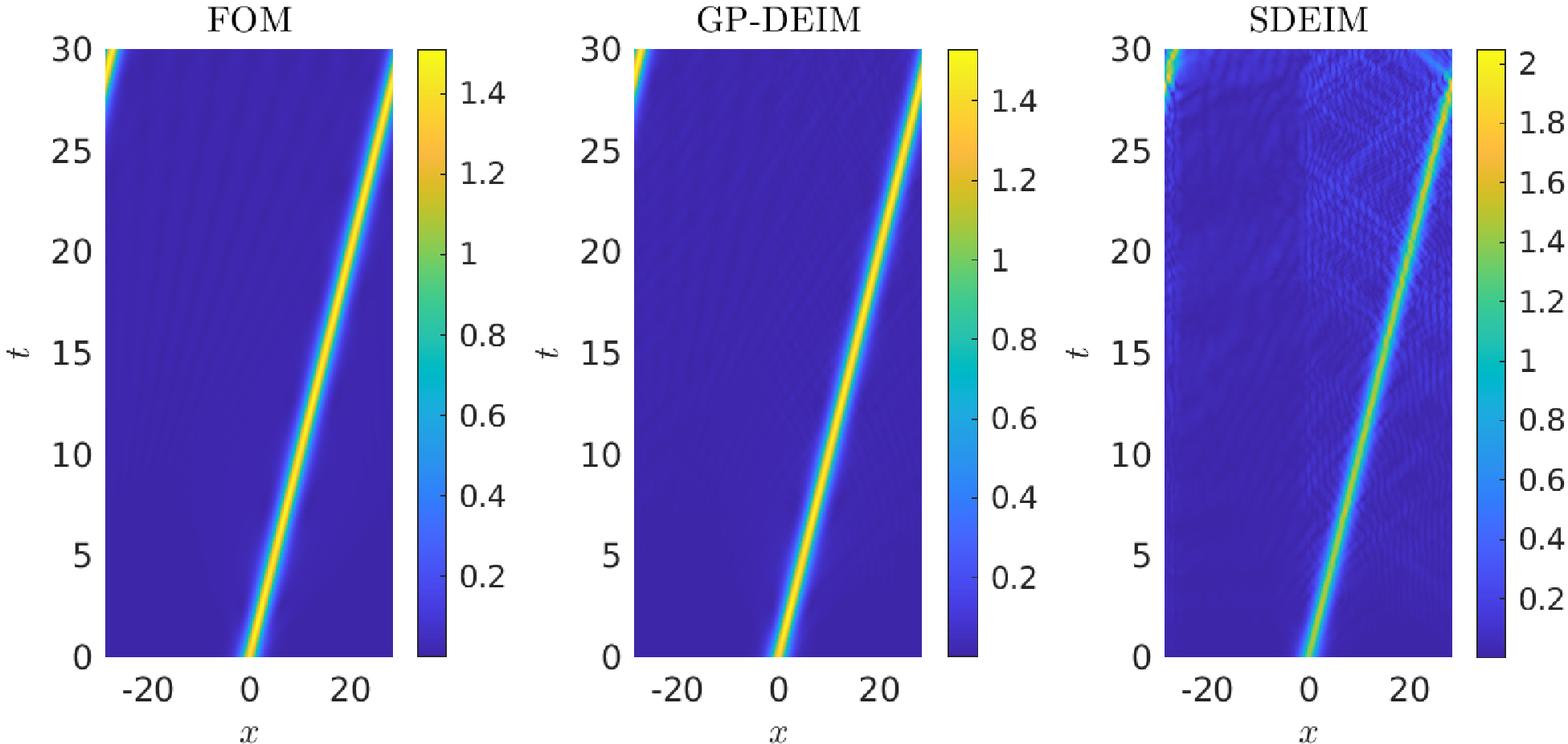}
    \caption{Numerical solution of the Schr\"odinger equation in space-time with $\prm = 1.0932$, for $\Nr=\nd=80$ (first row) and $\Nr=\nd=120$ (second row). The first column refers to the full order solution, the second column to the proposed gradient-preserving hyper-reduction, and the third column to the SDEIM of \cite{PM16,AH17}.}
    \label{fig:pcolor_n40_n60}
\end{figure}

In this problem, the dimension of the DEIM basis required by the hyper-reduced model to achieve the error of the reduced model can be rather large. This behavior can be ascribed to the fact that the Schr\"odinger equation and its Hamiltonian do not exhibit significant global reducibility properties.
\Cref{fig:singvalues} shows, on the left, the singular values of the full order solution obtained for $\nt=3000$ time steps and, on the right, the singular values of the snapshot matrix $M_J\in\Rbb^{\nf\times\Nr N_s}$ of the Jacobian for $N_s=150$ time instants (one every $20$ time steps). One value of the parameter is considered: $\prm=1.0932$.
Each line corresponds to a different dimension $\Nr$ of the reduced model.
We can observe that the singular values of the snapshots corresponding to the full order solution decay relatively slowly (\Cref{fig:singvalues}, left), which suggests that the dimension $\Nr$ of the reduced space needs to be sufficiently large to accurately represent the full dynamics. On the other hand, the reducibility of the Jacobian deteriorates as $\Nr$ increases (\Cref{fig:singvalues}, right).
This last remark also justifies our choice of performing hyper-reduction on the reduced Jacobian rather than on the full order one, as already pointed out in \Cref{rmk:FOMROMJ}.

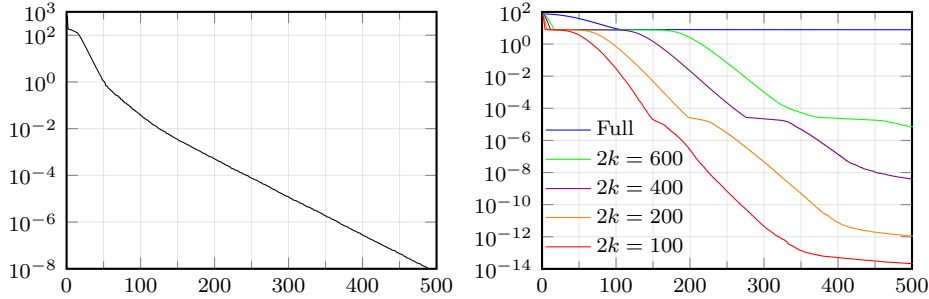
\begin{figure}[H]
\centering
\begin{tikzpicture}
    \begin{groupplot}[
      group style={group size=2 by 3,
                  horizontal sep=1.4cm},
      width=6.5cm, height=5cm
    ]
     \nextgroupplot[ylabel={},
                  xlabel={},
                  axis line style = thick,
                  grid=both,
                  minor tick num=1,
                  max space between ticks=50, ytick={10^(-8),10^(-6),10^(-4),10^(-2),1,10^(2),10^3},
                  xtick = {0,100,200,300,400,500},
                  grid style = {gray,opacity=0.2},
                  xmin=0, xmax=500,
                  ymax = 1e+3, ymin = 10^(-8),
                  ymode=log,
                  xlabel style={font=\footnotesize},
                  ylabel style={font=\footnotesize},
                  x tick label style={font=\footnotesize},
                  y tick label style={font=\footnotesize},
                  legend style={font=\footnotesize}]
        \addplot+[color=black,mark=none] table[x=svind,y=sv] {Sy_sv.txt};
        \nextgroupplot[ylabel={},
                  xlabel={},
                  axis line style = thick,
                  grid=both,
                  minor tick num=1,
                  max space between ticks=50, ytick={10^(-14),10^(-12),10^(-10),10^(-8),10^(-6),10^(-4),10^(-2),10^(0),10^(2)},
                  xtick = {0,100,200,300,400,500},
                  grid style = {gray,opacity=0.2},
                  xmin=0, xmax=500, 
                  ymax = 10^(2), ymin = 10^(-14),
                  ymode=log,
                  xlabel style={font=\footnotesize},
                  ylabel style={font=\footnotesize},
                  x tick label style={font=\footnotesize},
                  y tick label style={font=\footnotesize},
                  legend style={font=\footnotesize},
                  legend columns = 1,
                  legend cell align={left},
                  legend style={fill=none,draw=none,at={(0.42,0.62)}}]
        \addplot+[color=blue,mark=none] table[x=sv,y=full] {SJG_sv.txt};
        \addplot+[color=green,mark=none] table[x=sv,y=k300] {SJG_sv.txt};
        \addplot+[color=violet,mark=none] table[x=sv,y=k200] {SJG_sv.txt};
        \addplot+[color=orange,mark=none] table[x=sv,y=k100] {SJG_sv.txt};
        \addplot+[color=red,mark=none] table[x=sv,y=k50] {SJG_sv.txt};
        \legend{{Full},{$\Nr=600$},{$\Nr=400$},{$\Nr=200$},{$\Nr=100$}};
    \end{groupplot}
\end{tikzpicture}
\caption{NLS. Left: singular values of the snapshot matrix of the full order solution. Right: singular values of of the snapshot matrix $M_J\in\Rbb^{\nf\times\Nf N_s}$ of the Jacobian for different choices of $\nr$. Only the $500$ largest singular values are shown.}
\label{fig:singvalues}
\end{figure}

To address the lack of global reducibility of the Jacobian, we perform the adaptation of DEIM basis and indices described in \Cref{sec:adaptive_DEIM}. In this section we fix the size of the reduced basis to $\Nr=400$ and we consider $\prm=1.0932\notin\Sprmh$ as test parameter.

As a first test, we consider, at each update $j=1,\ldots,\nad$, the selection of the $\nd_s$ sampling points collected in the matrix $S_j$ and we compare three possible selection strategies:
(i) random;
(ii) based on the minimization of the residual norm $\normF{R_j}$ as proposed in \cite[Algorithm 1]{Peher20};
and (iii) based on the minimization of the projected residual norm $\normF{R_j\Cbb_j}$ motivated by \Cref{theo:rhoj_squared}, see also \Cref{alg:indices_adapt}.
We set the dimension of the DEIM basis to $\nd=100$, while the adaptation hyper-parameters are $\delta=\delta_0=\gamma=5$, $w=1$ and $r=\overline{r}$, namely a full-rank update is performed at each step. We compare the three selection strategies for two possible sizes of the sampling matrix, namely $\nd_s\in\{150,300\}$.
\Cref{fig:schro_samplingstrats} shows the error $\normF{\deimb_{j+1}C_j-F_j}$, at each update $j$, for the three selection criteria. It can be observed that the random algorithm gives the worst results, while the other two strategies do not yield significantly different results, at least in this simulation.

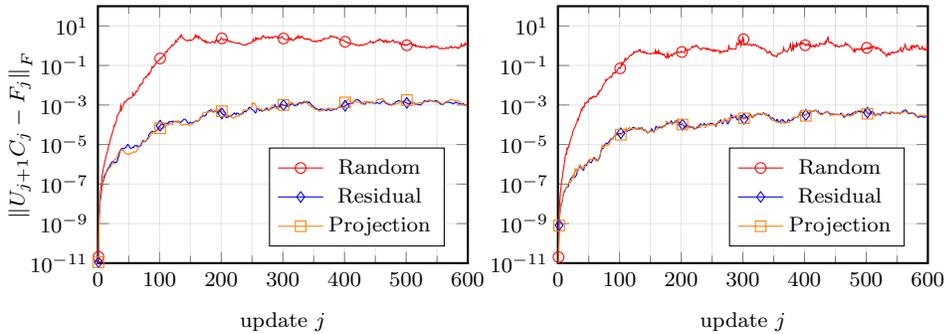
\begin{figure}[H]
\centering
\begin{tikzpicture}
    \begin{groupplot}[
      group style={group size=2 by 3,
                  horizontal sep=1.2cm},
      width=6.5cm, height=5cm
    ]
    \nextgroupplot[ylabel={$\normF{\deimb_{j+1}C_j-F_j}$},
                  ylabel style = {yshift=-.2cm},
                  xlabel={update $j$},
                  axis line style = thick,
                  grid=both,
                  minor tick num=1,
                  max space between ticks=20,
                  grid style = {gray,opacity=0.2},
                  xmin=0, xmax=600,
                  ymax = 10^(2), ymin = 10^(-11),
                  ymode=log,
                  ytick = {10^(-11),10^(-9),10^(-7),10^(-5),10^(-3),10^(-1),10},
                  xlabel style={font=\footnotesize},
                  ylabel style={font=\footnotesize},
                  x tick label style={font=\footnotesize},
                  y tick label style={font=\footnotesize},
                  legend style={font=\footnotesize},
                  legend columns = 1,
                  legend style={at={(0.7,0.45)},anchor=north}]
        \addplot+[color=red,mark=o,mark repeat=100,mark phase=1] table[x=adapt,y=random] {sc_samplingstrats_ms150.txt};
        \addplot+[color=blue,mark=diamond,mark repeat=100,mark phase=1] table[x=adapt,y=residual] {sc_samplingstrats_ms150.txt};
        \addplot+[color=orange,mark=square,mark repeat=100,mark phase=1] table[x=adapt,y=projection] {sc_samplingstrats_ms150.txt};
        \legend{{Random},
                {Residual},{Projection}
                };
        \nextgroupplot[ylabel={},
                  xlabel={update $j$},
                  axis line style = thick,
                  grid=both,
                  minor tick num=1,
                  max space between ticks=20,
                  grid style = {gray,opacity=0.2},
                  xmin=0, xmax=600,
                  ymax = 10^(2), ymin = 10^(-11),
                  ymode=log,
                  ytick = {10^(-11),10^(-9),10^(-7),10^(-5),10^(-3),10^(-1),10},
                  xlabel style={font=\footnotesize},
                  ylabel style={font=\footnotesize},
                  x tick label style={font=\footnotesize},
                  y tick label style={font=\footnotesize},
                  legend style={font=\footnotesize},
                  legend columns = 1,
                  legend style={at={(0.7,0.45)},anchor=north}]
        \addplot+[color=red,mark=o,mark repeat=100,mark phase=1] table[x=adapt,y=random] {sc_samplingstrats_ms300.txt};
        \addplot+[color=blue,mark=diamond,mark repeat=100,mark phase=1] table[x=adapt,y=residual] {sc_samplingstrats_ms300.txt};
        \addplot+[color=orange,mark=square,mark repeat=100,mark phase=1] table[x=adapt,y=projection] {sc_samplingstrats_ms300.txt}; 
        \legend{{Random},
                {Residual},{Projection}
                };
    \end{groupplot}
\end{tikzpicture}
\caption{NLS. Evolution of $\normF{\deimb_{j+1}C_j-F_j}$ for three different criteria of selection of the sampling points.
The number of sampling points is $\nd_s=150$ (left plot) and $\nd_s=300$ (right plot). The dimension of the reduced basis is $\Nr=400$ and the dimension of the DEIM basis is $\nd=100$.}
\label{fig:schro_samplingstrats}
\end{figure}

Next, we assess the adaptive hyper-reduction algorithm for different choices of the adaptation hyper-parameters -- $r$, $m_s$, $w$, $\delta$, and $\gamma$ -- in terms of: (i) accuracy of the solution, measured by the error $\Ecal_{L^2}^{\HR}$ from \eqref{eq:errNR}; (ii) computational time of both the offline phase, $t_{\off}$, and the online phase, $t_{\on}$, computed as the average time, in seconds, over $5$ independent runs; and (iii) conservation of the Hamiltonian via the error
$\Delta \Ham_{\nt}=\seminorm{\Ham(\fs^{\nt})-\Ham(\redb\hrs^{\nt})}$. We refer to \Cref{prop:ham_cons_adapt} for the analytic bound on $\Delta \Ham_{\nt}$; when AVF is used as numerical time integrator, the error sources in \eqref{eq:Hambound} that are independent on the reduction take values smaller than $10^{-10}$.

In all the following tests, we fix the size of the DEIM basis to $\nd=100$ and we let one hyper-parameter vary at a time. 
Based on the previous test and on \Cref{sec:adaptive_DEIM}, the sampling strategy in \Cref{alg:indices_adapt} is adopted for the update of the sampling indices.
As a reference, we include the performances of the non-adaptive hyper-reduction, where the offline phase consists of the construction of the DEIM basis from the matrix of the reduced Jacobian at full model snapshots obtained every $20$ time steps and $11$
training parameters.
If the adaptive scheme is employed instead, the offline phase involves solving the full system for $\delta_0\ll\nt$ time steps and only for the given test parameter.
This implies that
the offline phase will only require few seconds; on the other hand, the DEIM basis is adapted every $\delta$ time steps, which increases the computational cost of the online phase with respect to the non-adaptive scheme. 

\textbf{Effect of the rank $r$ of the DEIM update.}
We test the hyper-reduced model for different values of the rank $r$ of the DEIM update. Instead of fixing its value we let it change at each update $j$ based on the numerical rank of $S_j^\top R_jC_j^\top$ for a fixed tolerance $\tau_r$. The rank $r_j$ vs. the update number is shown in \Cref{fig:rj_and_msj} (left) for different choices of $\tau_r$, while the error $\Ecal_{L^2}^{\HR}$ of the hyper-reduced solution is reported in \Cref{tab:taur}. We observe that the error improves as $\tau_r$ decreases, while the cost of the online phase of the adaptive hyper-reduction algorithm is almost independent of $\tau_r$. This is due to the fact that the cost of the adaptation procedure is dominated by the computation of the DEIM residual matrix $R_j$ and the SVD of $R_j\Cbb_j$ in \Cref{alg:indices_adapt}, whose complexity is $O(\nh\nr\nd w)$, and thus independent of $r$. Moreover, in this experiment, there is no significant advantage in fixing the tolerance $\tau_r$ smaller than $10^{-8}$ and, correspondingly, performing a DEIM basis update of rank larger than $40$.
\begin{table}[H]
\footnotesize
\caption{{Error and computational times of the hyper-reduced model for different rank tolerances $\tau_r$.}}\label{tab:taur}
    \centering
    \begin{tabular}{cc|ccc|c|c|}\cline{3-7}
    & & $t_{\off}$ & $t_{\on}$ & $t_{\tot}$ & $\Ecal_{L^2}^{\HR}$ & $\Delta \Ham_{\nt}$ \\
    \hline
    \multicolumn{2}{|c|}{Non-adaptive} & 152.77 & 38.52 & 191.29 & 1.04e+00 & 2.03e+00 \\\hline
    \multicolumn{1}{|c}{\multirow{5}{*}{\begin{tabular}[c]{@{}c@{}}$\nd_s=250$\\ $\delta_0=\delta=5$\\ $w=1$\\$\gamma=\delta$\end{tabular}}} & $\tau_r=10^{-4}$ & 5.44 & 54.83 & 60.27 & 3.04e-03 & 7.09e-02 \\
    \multicolumn{1}{|c}{} & $\tau_r=10^{-6}$ & 5.44 & 55.17 & 60.61 & 1.09e-04 & 2.37e-04 \\
    \multicolumn{1}{|c}{} & $\tau_r=10^{-8}$ & 5.44 & 55.30 & 60.74 & 6.44e-05 & 6.88e-04 \\
    \multicolumn{1}{|c}{} & $\tau_r=10^{-10}$ & 5.44 & 55.46 & 60.90 & 6.17e-05 & 8.04e-04 \\
    \multicolumn{1}{|c}{} & $\tau_r=10^{-12}$ & 5.44 & 55.58 & 61.02 & 6.03e-05 & 7.21e-04 \\ \hline
    \end{tabular}
\end{table}

\textbf{Effect of the number $\nd_s$ of sampling points.}
Analogously to the rank of the update, instead of fixing $\nd_s$ we may select the number of sampling indices based on a tolerance $\tau_s$. More precisely, an index is chosen at the $j$th adaptation step if the $2$-norm of the corresponding row of $R_j\Cbb_j$ is larger than $\tau_s$. In this way, the number of sampling indices changes at each adaptation step and it is fixed to be at least equal to $\nd$.
The results are reported in \Cref{tab:taus}.
The number of sampling indices selected at each adaptation for different choices of the tolerance is shown in \Cref{fig:rj_and_msj} (right). As $\nd_s$ increases, the error improves while the computational time is almost independent of the value of $\nd_s$, and the total time required by the adaptive algorithm is less than one third of the time needed by the non-adaptive algorithm.
\begin{table}[H]
\footnotesize
\caption{Error and computational times of the hyper-reduced model for different sampling tolerances~$\tau_s$.}\label{tab:taus}
    \centering
    \begin{tabular}{cc|ccc|c|c|}\cline{3-7}
     & & $t_{\off}$ & $t_{\on}$ & $t_{\tot}$ & $\Ecal_{L^2}^{\HR}$ & $\Delta \Ham_{\nt}$ \\
     \hline
    \multicolumn{2}{|c|}{Non-adaptive} & 152.77 & 38.52 & 191.29 & 1.04e+00 & 2.03e+00 \\\hline
    \multicolumn{1}{|c}{\multirow{5}{*}{\begin{tabular}[c]{@{}c@{}}$\delta_0=\delta=5$\\ $w=1$\\ $r=\overline{r}$\\$\gamma=\delta$\end{tabular}}} & $\tau_s=10^{-4}$ & 5.50 & 53.32 & 58.82 & 3.94e-04 & 1.12e-02 \\
    \multicolumn{1}{|c}{} & $\tau_s=10^{-6}$ & 5.50 & 53.81 & 59.31 & 3.94e-04 & 1.12e-02 \\
    \multicolumn{1}{|c}{} & $\tau_s=10^{-8}$ & 5.50 & 53.82 & 59.32 & 2.86e-04 & 5.18e-03 \\
    \multicolumn{1}{|c}{} & $\tau_s=10^{-10}$ & 5.50 & 54.42 & 59.92 & 5.99e-05 & 9.80e-05 \\
    \multicolumn{1}{|c}{} & $\tau_s=10^{-12}$ & 5.50 & 55.49 & 60.99 & 4.70e-05 & 1.52e-04 \\ \hline
    \end{tabular}
\end{table}

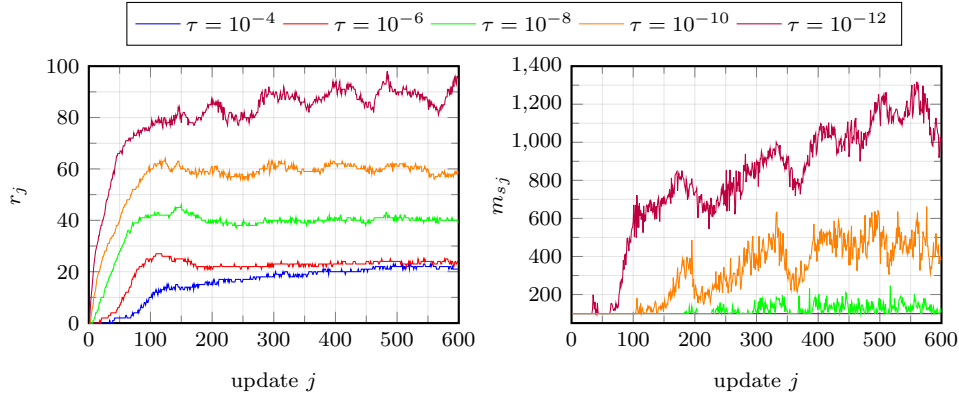
\begin{figure}[H]
\centering
\begin{tikzpicture}
    \begin{groupplot}[
      group style={group size=2 by 3,
                  horizontal sep=1.5cm},
      width=6.5cm, height=5cm
    ]
      \nextgroupplot[ylabel={$r_j$},
                  xlabel={update $j$},
                  ylabel style = {yshift=-.3cm},
                  axis line style = thick,
                  grid=both,
                  minor tick num=1,
                  max space between ticks=20,
                  grid style = {gray,opacity=0.2},
                  xmin=0, xmax=600,
                  ymin = 0, ymax = 100,
                  xlabel style={font=\footnotesize},
                  ylabel style={font=\footnotesize},
                  x tick label style={font=\footnotesize},
                  y tick label style={font=\footnotesize},
                  legend style={font=\footnotesize},
                  legend columns = 5,
                  legend style={at={(1.15,1.25)},anchor=north}]
        \addplot+[color=blue,mark=none] table[x=adapt,y=taum4] {rank_adapt_sc.txt};
        \addplot+[color=red,mark=none] table[x=adapt,y=taum6] {rank_adapt_sc.txt};
        \addplot+[color=green,mark=none] table[x=adapt,y=taum8] {rank_adapt_sc.txt};
        \addplot+[color=orange,mark=none] table[x=adapt,y=taum10] {rank_adapt_sc.txt};
        \addplot+[color=purple,mark=none] table[x=adapt,y=taum12] {rank_adapt_sc.txt};
        \legend{{$\tau=10^{-4}$},{$\tau=10^{-6}$},{$\tau=10^{-8}$},{$\tau=10^{-10}$},{$\tau=10^{-12}$}};
        \nextgroupplot[ylabel={${\nd_s}_j$},
                  xlabel={update $j$},
                  ylabel style = {yshift=-.3cm},
                  axis line style = thick,
                  grid=both,
                  minor tick num=1,
                  max space between ticks=20,
                  grid style = {gray,opacity=0.2},
                  xmin=0, xmax=600,
                  ymin = 50, ymax = 1400,
                  ytick = {200,400,600,800,1000,1200,1400},
                  xlabel style={font=\footnotesize},
                  ylabel style={font=\footnotesize},
                  x tick label style={font=\footnotesize},
                  y tick label style={font=\footnotesize},
                  legend style={font=\footnotesize},
                  legend columns = 2,
                  legend style={at={(0.5,1.45)},anchor=north}]
        \addplot+[color=blue,mark=none] table[x=adapt,y=taum4] {ms_adapt_sc.txt};
        \addplot+[color=orange,mark=none] table[x=adapt,y=taum10] {ms_adapt_sc.txt};
        \addplot+[color=red,mark=none] table[x=adapt,y=taum6] {ms_adapt_sc.txt};
        \addplot+[color=purple,mark=none] table[x=adapt,y=taum12] {ms_adapt_sc.txt};
        \addplot+[color=green,mark=none] table[x=adapt,y=taum8] {ms_adapt_sc.txt};
    \end{groupplot}
\end{tikzpicture}
\caption{NLS. Left: Rank $r_j=\text{rank}(S_j^\top R_jC_j^\top)$ at every update $j$ for different values of $\tau_t=\tau$. Right: Number of sampling indices ${\nd_s}_j$ at every update $j$ for different values of $\tau_s=\tau$.} 
\label{fig:rj_and_msj}
\end{figure}

\textbf{Effect of the size $w$ of the time window to evaluate the nonlinear term.}
The role of $w$ in the adaptation is described in \Cref{sec:adaptive_DEIM}.
In this test, we slightly increase the values of $\delta_0$ and $\delta$ and set them equal to $10$. The reason is that it must hold $w<\delta_0$, and, moreover, 
it is reasonable to require that $w<\delta$, that is, at each adaptation step, the temporal window only contains snapshots obtained after the last update.
In \Cref{tab:w} we observe that, as expected, the error $\Ecal_{L^2}^{\HR}$ decreases as more information from the past is included in the construction of the DEIM basis update. In parallel, the increase of the window size requires larger computational times in the online phase.
However, it can be noticed that the online phase in the case $w=1$ is only slightly more computationally expensive than the online phase of the non-adaptive algorithm, while the $L^2$-error is reduced by almost four orders of magnitude.
\begin{table}[H]
\footnotesize
\caption{Error and computational times of the hyper-reduced model for different window sizes $w$.}\label{tab:w}
    \centering
    \begin{tabular}{cc|ccc|c|c|}\cline{3-7}
     & & $t_{\off}$ & $t_{\on}$ & $t_{\tot}$ & $\Ecal_{L^2}^{\HR}$ & $\Delta \Ham_{\nt}$ \\ \hline
    \multicolumn{2}{|c|}{Non-adaptive} & \footnotesize 152.77 & 38.52 & 191.29 & 1.04e+00 & 2.03e+00 \\\hline
    \multicolumn{1}{|c}{\multirow{4}{*}{\begin{tabular}[c]{@{}c@{}}$\nd_s=250$\\ $\delta_0=\delta=10$\\ $r=\overline{r},\, \gamma=\delta$\end{tabular}}} & $w=1$ & 8.91 & 45.01 & 53.92 & 1.31e-04 & 2.55e-03 \\
    \multicolumn{1}{|c}{} & $w=2$ & 8.91 & 50.82 & 59.73 & 1.13e-04 & 2.67e-03  \\
    \multicolumn{1}{|c}{} & $w=4$ & 8.91 & 60.13 & 69.04 & 7.44e-05 & 1.30e-04  \\
    \multicolumn{1}{|c}{} & $w=6$ & 8.91 & 69.73 & 78.64 & 6.91e-05 & 1.06e-03 \\ \hline
    \end{tabular}
\end{table}

\textbf{Effect of the number $\delta$ of time steps between each DEIM adaptation.}
Predictably, the online phase is cheaper for large $\delta$ since the adaptations are less frequent, while the error increases, as reported in \Cref{tab:delta}.

\begin{table}[H]
\footnotesize
\caption{Error and computational times of the hyper-reduced model. Effect of the number $\delta$ of time steps  between each update of the DEIM basis.}\label{tab:delta}
    \centering
    \begin{tabular}{cc|ccc|c|c|}\cline{3-7}
     & & $t_{\off}$ & $t_{\on}$ & $t_{\tot}$ & $\Ecal_{L^2}^{\HR}$ & $\Delta \Ham_{\nt}$ \\ \hline
    \multicolumn{2}{|c|}{Non-adaptive} & 152.77 & 38.52 & 191.29 & 1.04e+00 & 2.03e+00 \\\hline
   \multicolumn{1}{|c}{\multirow{3}{*}{\begin{tabular}[c]{@{}c@{}}$\nd_s=250$\\ $\delta_0=5,\,w=1$\\ $r=\overline{r},\,\gamma=\delta$\end{tabular}}} & $\delta=5$ & 5.55 & 55.05 & 60.60 & 6.16e-05 & 5.31e-04 \\
   \multicolumn{1}{|c}{} & $\delta=10$ & 5.55 & 46.40 & 51.95 & 1.40e-04 & 3.08e-03 \\
   \multicolumn{1}{|c}{} & $\delta=20$ & 5.55 & 41.00 & 46.55 & 2.50e-03 & 4.99e-02 \\ \hline
    \end{tabular}
\end{table}

\textbf{Effect of the number of time steps $\gamma$ between each update of the sampling indices.}
In view of the discussion on the arithmetic complexity of sampling in \Cref{sec:samp_points_update}, it might seem preferable, in terms of computational cost, to update the sampling indices every $\gamma>1$ time steps. 
Indeed, the larger $\gamma$, the less frequently \Cref{alg:indices_adapt} is run, including the expensive computation of the full residual.
In this test we consider $\gamma=\nu\delta$, i.e. the sampling points are updated every $\nu$ DEIM adaptations. 
We observe in \Cref{tab:gamma} that the error 
of the hyper-reduced solution decreases as $\gamma$ is halved, with an increase in the runtime of the online phase of less than $5$ seconds.
\begin{table}[H]
\footnotesize
\caption{Error and computational times of the hyper-reduced model. Effect of the number $\gamma$ of time steps  between each update of the sampling indices.}\label{tab:gamma}
    \centering
    \begin{tabular}{cc|ccc|c|c|}\cline{3-7}
     & & $t_{\off}$ & $t_{\on}$ & $t_{\tot}$ & $\Ecal_{L^2}^{\HR}$ & $\Delta \Ham_{\nt}$ \\ \hline
    \multicolumn{2}{|c|}{Non-adaptive} & 152.77 & 38.52 & 191.29 & 1.04e+00 & 2.03e+00 \\\hline
    \multicolumn{1}{|c}{\multirow{4}{*}{\begin{tabular}[c]{@{}c@{}}$\nd_s=250$\\$\delta_0=\delta=5$\\ $w=1$\\ $r=\overline{r}$\end{tabular}}} & $\gamma=5$ & 5.51 & 55.15 & 60.66 & 6.16e-05 & 5.31e-04 \\
    \multicolumn{1}{|c}{} & $\gamma=10$ & 5.51 & 50.35 & 55.86 & 1.11e-04 & 2.04e-03 \\
    \multicolumn{1}{|c}{} & $\gamma=20$ & 5.51 & 48.20 & 53.71 & 8.59e-04 & 1.83e-02 \\
    \multicolumn{1}{|c}{} & $\gamma=30$ & 5.51 & 47.56 & 53.07 & 4.58e-03 & 9.90e-02 \\ \hline
    \end{tabular}
\end{table}

Although the behavior of the adaptive hyper-reduction algorithm under variation of the hyper-parameters is clearly problem-dependent, the numerical tests presented above suggest the following conclusions. Increasing the rank $r$ of the update and the number $\nd_s$ of sampling points leads to a significant reduction of the error without any major effect on the computational time. On the other hand, the optimal values of $r$ and $\nd_s$ may depend on the adaptation step $j$ and, hence, it seems preferable to let them vary in time according to a tolerance. The effect of $w$ on the error is less evident: although the error decreases as $w$ increases, the fast growth of the computational time $t_{\on}$ outweighs this gain.
Finally, small values of $\delta$ and $\gamma$ imply frequent adaptations of both the DEIM basis and sampling indices, and thus reduction in the error, at a mild increase of the computational cost of the simulation.

As a last numerical test, we fix the values of the
hyper-parameters and compare the performances of the non-adaptive and adaptive hyper-reduction algorithms as the dimension $\nd$ of the DEIM space varies, for different choices of the reduced dimension $\Nr$.
We set the hyper-parameters as follows: the tolerance in the choice of the update rank is $\tau_r=10^{-12}$; the tolerance in the selection of the sampling points is $\tau_s=10^{-10}$; the temporal window size is $w=1$; the frequency of the updates is $\delta=\gamma=5$.
In \Cref{fig:error_vs_ctime_2k400} we report the error $\Ecal_{L^2}^{\HR}$ between the hyper-reduced solution and the full model solution vs. $\nd$ (left plot) and vs. the algorithm runtime (right plot) for $\Nr=400$.
In \Cref{fig:error_vs_ctime_2k}, we study the computational cost of the hyper-reduction algorithms for two more dimensions of the reduced space, namely for $\Nr\in\{160,280\}$. The reported times are computed as averages over $5$ runs.

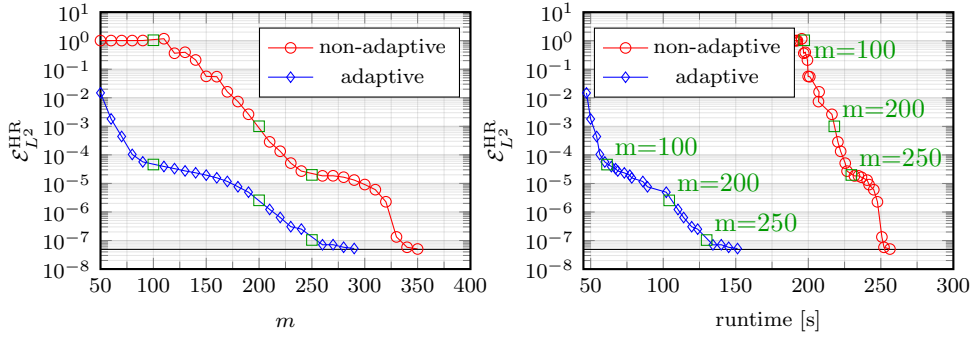
\begin{figure}[H]
\centering
\begin{tikzpicture}
    \begin{groupplot}[
      group style={group size=2 by 3,
                  horizontal sep=1.5cm},
      width=6.5cm, height=5cm
    ]
      \nextgroupplot[ylabel={$\Ecal_{L^2}^{\HR}$},
                  xlabel={$\nd$},
                  ylabel style = {yshift=-.2cm},
                  axis line style = thick,
                  grid=both,
                  minor tick num=1,
                  max space between ticks=20,
                  grid style = {gray,opacity=0.2},
                  xmin=50, xmax=400,
                  ymin = 1e-08, ymax = 1e+01,
                  ymode=log,
                  ytick = {1e-08,1e-07,1e-06,1e-05,1e-04,1e-03,1e-02,1e-01,1e+00,1e+01},
                  xlabel style={font=\footnotesize},
                  ylabel style={font=\footnotesize},
                  x tick label style={font=\footnotesize},
                  y tick label style={font=\footnotesize},                  
                  legend style={font=\footnotesize},
                  legend columns = 1,
                  legend style={at={(0.7,0.94)},anchor=north}]
        \addplot+[mark=o,color=red] [select coords between index={0}{27}] table[x=m,y=errnonad] {error_vs_m_and_ctime.txt};
        \addplot+[mark=diamond,color=blue] [select coords between index={0}{21}] table[x=m,y=errad] {error_vs_m_and_ctime.txt};
        \addplot+[mark=square,color=green!60!black,only marks] table[x=mnonad,y=errnonad] {error_vs_m_diffmarks.txt};
        \addplot+[mark=square,color=green!60!black,only marks] table[x=mad,y=errad] {error_vs_m_diffmarks.txt};
        \addplot+[color=black,mark=none] table[x=m,y=errred] {error_vs_m_and_ctime.txt};
        \legend{{non-adaptive}, {adaptive}};
        \nextgroupplot[ylabel={$\Ecal_{L^2}^{\HR}$},
                  xlabel={runtime [s]},
                  xlabel style = {yshift=.1cm},
                  ylabel style = {yshift=-.1cm},
                  axis line style = thick,
                  grid=both,
                  minor tick num=1,
                  max space between ticks=20,
                  grid style = {gray,opacity=0.2},
                  xmin=45, xmax=300,
                  ymin = 1e-08, ymax = 1e+01,
                  ymode = log,
                  ytick = {1e-08,1e-07,1e-06,1e-05,1e-04,1e-03,1e-02,1e-01,1e+00,1e+01},
                  xlabel style={font=\footnotesize},
                  ylabel style={font=\footnotesize},
                  x tick label style={font=\footnotesize},
                  y tick label style={font=\footnotesize},
                  legend style={font=\footnotesize},
                  legend columns = 1,
                  legend style={at={(0.3,0.94)},anchor=north}]
        \addplot+[mark=o,color=red] [select coords between index={0}{27}] table[x=tnonad,y=errnonad] {error_vs_m_and_ctime.txt};
        \addplot+[mark=diamond,color=blue] [select coords between index={0}{21}] table[x=tad,y=errad] {error_vs_m_and_ctime.txt};
        \addplot+[mark=square,color=green!60!black,nodes near coords,point meta=explicit symbolic,only marks,visualization depends on=\thisrow{alignmentNA} \as \alignment, nodes near coords,point meta=explicit symbolic,every node near coord/.style={anchor=\alignment}] table[x=tnonad,y=errnonad,meta index=4] {error_vs_ctime_diffmarks.txt};
        \addplot+[mark=square,color=green!60!black,nodes near coords,point meta=explicit symbolic,only marks,visualization depends on=\thisrow{alignmentA} \as \alignment, nodes near coords,point meta=explicit symbolic,every node near coord/.style={anchor=\alignment}] table[x=tad,y=errad,meta index=4] {error_vs_ctime_diffmarks.txt};
        \addplot+[color=black,mark=none] table[x=xred,y=errred] {error_vs_ctime_diffmarks.txt};
        \legend{{non-adaptive},
          {adaptive}};
    \end{groupplot}
\end{tikzpicture}
\vspace{-1em}
\caption{NLS. Reduced model of size $\Nr=400$.
Left: error $\Ecal_{L^2}^{\HR}$ vs. DEIM size $\nd$.  Right: $\Ecal_{L^2}^{\HR}$ vs. total algorithm runtime for different values of $\nd$. The black line represents the error between the reduced solution and the full solution. The computational time to solve the reduced model is $230$\,s. The computational time of the offline phase of the hyper-reduction is $152.8$\,s in the non-adaptive case, $5.5$\,s in the adaptive case.}\label{fig:error_vs_ctime_2k400}
\end{figure}
\begin{figure}[H]
\centering
\begin{tikzpicture}
    \begin{groupplot}[
      group style={group size=2 by 3,
                  horizontal sep=1.5cm},
      width=6.5cm, height=5cm
    ]
              \nextgroupplot[ylabel={$\Ecal_{L^2}^{\HR}$},
                  xlabel={runtime [s]},
                  xlabel style = {yshift=.1cm},
                  ylabel style = {yshift=-.1cm},
                  axis line style = thick,
                  title = {$\Nr=160$},
                  title style={font=\footnotesize},
                  grid=both,
                  minor tick num=1,
                  max space between ticks=20,
                  grid style = {gray,opacity=0.2},
                  xmin=5, xmax=90,
                  ymin = 1e-03, ymax = 1e+01,
                  ymode = log,
                  ytick = {1e-03,1e-02,1e-01,1e+00,1e+01},
                  xlabel style={font=\footnotesize},
                  ylabel style={font=\footnotesize},
                  x tick label style={font=\footnotesize},
                  y tick label style={font=\footnotesize},
                  legend style={font=\footnotesize},
                  legend columns = 1,
                  legend style={at={(0.3,0.96)},anchor=north}]
        \addplot+[mark=o,color=red] [select coords between index={0}{9}] table[x=tnonad,y=errnonad] {error_vs_m_and_ctime_2k160.txt};
        \addplot+[mark=diamond,color=blue] [select coords between index={16}{16}] table[x=tad,y=errad] {error_vs_m_and_ctime_2k160.txt};
        \addplot+[color=black,mark=none] table[x=xred,y=errred] {error_vs_ctime_diffmarks_2k160.txt};
        \addplot+[mark=square,color=green!60!black,nodes near coords,point meta=explicit symbolic,only marks,visualization depends on=\thisrow{alignmentNA} \as \alignment, nodes near coords,point meta=explicit symbolic,every node near coord/.style={anchor=\alignment}] [select coords between index={2}{4}]
        table[x=tnonad,y=errnonad,meta index=4] {error_vs_ctime_diffmarks_2k160.txt};
        \addplot+[mark=square,mark options={green!60!black},nodes near coords,point meta=explicit symbolic,only marks,visualization depends on=\thisrow{alignmentA} \as \alignment, nodes near coords,point meta=explicit symbolic,every node near coord/.style={anchor=\alignment}] [select coords between index={0}{1}] table[x=tad,y=errad,meta index=5] {error_vs_ctime_diffmarks_2k160.txt};
        \addplot+[color=blue,solid,mark=none] [select coords between index={11}{14}] table[x=tad,y=errad] {error_vs_m_and_ctime_2k160.txt};
        \addplot+[color=red,solid,mark=none] [select coords between index={0}{10}] table[x=tnonad,y=errnonad] {error_vs_m_and_ctime_2k160.txt};
        \legend{{non-adaptive}, {adaptive}};
        \nextgroupplot[ylabel={$\Ecal_{L^2}^{\HR}$},
                  xlabel={runtime [s]},
                  xlabel style = {yshift=.1cm},
                  ylabel style = {yshift=-.1cm},
                  title = {$\Nr=280$},
                  title style={font=\footnotesize},
                  axis line style = thick,
                  grid=both,
                  minor tick num=1,
                  max space between ticks=20,
                  grid style = {gray,opacity=0.2},
                  xmin=15, xmax=150,
                  ymin = 1e-06, ymax = 1e+01,
                  ymode = log,
                  ytick = {1e-06,1e-05,1e-04,1e-03,1e-02,1e-01,1e+00,1e+01},
                  xlabel style={font=\footnotesize},
                  ylabel style={font=\footnotesize},
                  x tick label style={font=\footnotesize},
                  y tick label style={font=\footnotesize},
                  legend style={font=\footnotesize},
                  legend columns = 1,
                  legend style={at={(0.3,0.95)},anchor=north}]
        \addplot+[mark=o,color=red] [select coords between index={0}{18}] table[x=tnonad,y=errnonad] {error_vs_m_and_ctime_2k280.txt};
        \addplot+[mark=diamond,color=blue] [select coords between index={0}{9}] table[x=tad,y=errad] {error_vs_m_and_ctime_2k280.txt};
        \addplot+[mark=square,color=green!60!black,nodes near coords,point meta=explicit symbolic,only marks,visualization depends on=\thisrow{alignmentNA} \as \alignment, nodes near coords,point meta=explicit symbolic,every node near coord/.style={anchor=\alignment}] table[x=tnonad,y=errnonad,meta index=4] {error_vs_ctime_diffmarks_2k280.txt};
        \addplot+[mark=square,color=green!60!black,nodes near coords,point meta=explicit symbolic,only marks,visualization depends on=\thisrow{alignmentA} \as \alignment, nodes near coords,point meta=explicit symbolic,every node near coord/.style={anchor=\alignment}] [select coords between index={0}{2}] table[x=tad,y=errad,meta index=4] {error_vs_ctime_diffmarks_2k280.txt};
        \addplot+[color=black,mark=none] table[x=xred,y=errred] {error_vs_ctime_diffmarks_2k280.txt};
        \legend{{non-adaptive},
          {adaptive}};
    \end{groupplot}
\end{tikzpicture}
\caption{NLS. Hyper-reduction error $\Ecal_{L^2}^{\HR}$ vs. total algorithm runtime for different values of $\nd$. Reduced model of size $\Nr=160$ (left) and $\Nr=280$ (right).
For $\Nr=160$: the computational time to solve the reduced model is $66$\,s; the offline phase of the hyper-reduction takes $69.7$\,s in the non-adaptive case and $0.5$\,s in the adaptive case.
For $\Nr=280$: the computational time to solve the reduced model is $126$\,s;
the offline phase of the hyper-reduction takes $109.6$\,s in the non-adaptive case and $1.2$\,s in the adaptive case.}
\label{fig:error_vs_ctime_2k}
\end{figure}
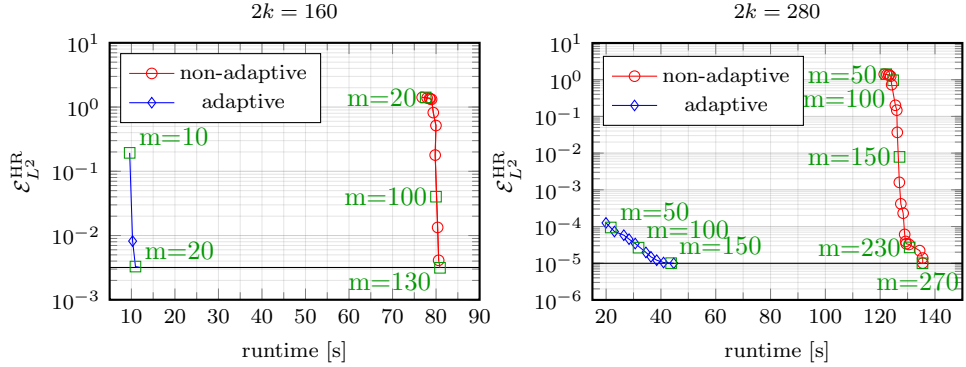

\Cref{fig:error_vs_ctime_2k400,fig:error_vs_ctime_2k}
show that, to achieve a given error, the adaptive scheme allows a smaller DEIM basis compared to the non-adaptive one. This translates in a reduced computational cost, despite the fact that the adaptive algorithm requires extra online operations to perform the DEIM update. This gain in the total runtime is also due to the fact that a more expensive online phase for adaptation is compensated by a cheaper offline phase compared to the non-adaptive strategy. As an example, the total runtime is reduced by a factor $3$ for $\nd=100$ when $\Nr=400$, and by a factor $4$ when $\Nr=280$. 
Moreover, for a given DEIM size $\nd$, the adaptive scheme yields smaller errors than the non-adaptive one: for example, for $\nd=200$, the error in the adaptive case is almost three orders of magnitude lower when $\Nr=400$, and a similar gain is observed for $\nd=150$ when $\Nr=280$.
Finally, when the dimension of the reduced space is $\Nr=400$, the hyper-reduced model achieves the error of the reduced model for $\nd=290$ in the adaptive case, and for $\nd=350$ in the non-adaptive one; for these values of $\nd$, the adaptive algorithm requires about $40\%$ less runtime than the non-adaptive scheme. This gain in computational efficiency becomes more evident when the reduced space is smaller. For $\Nr=280$, the DEIM modes required to achieve the reference error in the adaptive case are about half of those needed in the non-adaptive case ($160$ vs. $270$), with an approximately $70\%$ saving of the total computational time. With an even smaller reduced space, $\Nr=160$, the adaptive hyper-reduced model converges to the reduced model error for $m=25$, compared to $m=130$ in the non-adaptive case, and the total runtime is about $85\%$ smaller; even comparing the online phases alone, solving the adaptive scheme with $m=20$ is about as cheap as solving the non-adaptive scheme for $m=80$, while the error is over $2$ orders of magnitude smaller.
Finally, these numerical experiments show that the beneficial properties of the proposed adaptive strategy become more apparent when the hyper-reduction is performed on a relatively small reduced space. This is related to the decay of the singular values of the reduced Jacobian highlighted in \Cref{fig:singvalues}.

\section{Concluding remarks}\label{sec:conclusions}
We have proposed a hyper-reduction method to deal with large-scale parametric dynamical systems characterized by nonlinear gradient fields.
Its combination with symplectic reduced basis methods enables efficient numerical simulations of Hamiltonian systems with general nonlinearities. Indeed the resulting hyper-reduced models can be solved at a computational cost independent of the size of the full model and without compromising the physical properties of the dynamics thanks to the preservation of the gradient structure of the velocity field.
Although we have focused on Hamiltonian dynamical systems, the method can be applied to any gradient field, in e.g. gradient flows or in port-Hamiltonian problems.

Some questions remain open.
A rigorous connection between the reducibility properties of the solution space, of the space of nonlinear gradient fields and of the space of reduced Jacobian matrices is not understood yet. This could provide useful insights on when the simulation of a system can significantly benefit from hyper-reduction and whether an adaptive approach is to be preferred over a global one.

Moreover, our hyper-reduction strategy hinges on the properties of the reduced space. We observe that a larger reduced space improves the approximation of the state but demands a larger DEIM basis for accurate hyper-reduction. A natural idea to address this shortcoming is to consider symplectic reduced spaces that can evolve in time, as done e.g. in \cite{P19}. This would allow us to deal with small reduced spaces and hence highly reducible Jacobian matrices (i.e. small DEIM bases as well). The combination of evolving reduced spaces with adaptive hyper-reduction provides an interesting possible direction of future investigation.

\section*{Acknowledgments}
CP would like to thank Benjamin Peherstorfer for interesting discussions on adaptive DEIM preliminary to this work.

\bibliographystyle{siamplain}
\bibliography{references}
\end{document}